\documentclass[12pt]{amsart}

\usepackage{amssymb}
\usepackage{amsmath}
\usepackage{verbatim}
\usepackage[all]{xy}
\usepackage[hmargin=1.25in,top=1.5in,bottom=1.25in]{geometry}
\usepackage{mathrsfs}
\usepackage{graphicx}
\usepackage{combelow}
\usepackage[small,nohug]{diagrams}
\diagramstyle[labelstyle=\scriptstyle]

\makeatletter

\newcommand{\Rmnum}[1]{\expandafter\@slowromancap\romannumeral #1@}
\makeatother

\newcommand{\id}{\operatorname{id}}

\newcommand{\Hom}{\operatorname{Hom}}
\newcommand{\Aut}{\operatorname{Aut}}

\newcommand{\ch}{\operatorname{char}}

\newcommand{\Spec}{\operatorname{Spec}}

\newcommand{\uIsom}{\underline{\operatorname{Isom}}}
\newcommand{\bIsom}{\operatorname{\mathbf{Isom}}}
\newcommand{\Proj}{\operatorname{Proj}}

\newcommand{\InfAut}{\operatorname{InfAut}}
\newcommand{\Def}{\operatorname{Def}}
\newcommand{\Obs}{\operatorname{Obs}}
\newcommand{\Sch}{\text{Sch}}

\newcommand{\Pic}{\text{Pic}}

\newcommand{\etale}{\'{e}tale}

\newcommand{\Mustata}{Musta\cb{t}\u{a}}

\newcommand{\CC}{\mathbb{C}}

\newcommand{\QQ}{\mathbb{Q}}
\newcommand{\ZZ}{\mathbb{Z}}

\newcommand{\PP}{\mathbb{P}}

\newcommand{\bA}{\mathbb{A}}

\newcommand{\sA}{\mathcal{A}}
\newcommand{\sB}{\mathcal{B}}
\newcommand{\sC}{\mathcal{C}}
\newcommand{\sD}{\mathcal{D}}

\newcommand{\sH}{\mathcal{H}}
\newcommand{\sM}{\mathcal{M}}
\newcommand{\sMbar}{\overline{\mathcal{M}}}
\newcommand{\Mbar}{\overline{M}}

\newcommand{\sO}{\mathcal{O}}

\newcommand{\sL}{\mathcal{L}}
\newcommand{\sS}{\mathcal{S}}

\newcommand{\sU}{\mathcal{U}}
\newcommand{\sV}{\mathcal{V}}
\newcommand{\sR}{\mathcal{R}}
\newcommand{\sZ}{\mathcal{Z}}

\newcommand{\gt}{\mathfrak{t}}

\newcommand{\ra}{\rightarrow}

\newcommand{\hra}{\hookrightarrow}
\newcommand{\ol}{\overline}

\newcommand{\abs}[1]{| #1 | }

\newcommand{\ras}[1]{ \stackrel{ #1 }{\ra} }

\newtheorem{theorem}{Theorem}[section]

\newtheorem{lemma}[theorem]{Lemma}
\newtheorem{corollary}[theorem]{Corollary}

\theoremstyle{definition}
\newtheorem{definition}[theorem]{Definition}
\newtheorem{lemma-definition}[theorem]{Lemma-Definition}
\newtheorem{remark}[theorem]{Remark}

\begin{document}

\title[Alternate compactifications of the moduli space of genus one maps]{Alternate compactifications of the moduli space of genus one maps}
\date{}

\author{Michael Viscardi}
\email{viscardi@post.harvard.edu}

\begin{abstract}
We extend the definition of an $m$-stable curve introduced by Smyth to the setting of maps to a projective variety $X$, generalizing the definition of a Kontsevich stable map in genus one.  We prove that the moduli problem of $n$-pointed $m$-stable genus one maps of class $\beta$ is representable by a proper Deligne-Mumford stack $\sMbar_{1,n}^m(X,\beta)$ over $\Spec \ZZ[1/6]$.  For $X = \PP^r$, we explicitly describe all of the irreducible components of $\sMbar_{1,n}(\PP^r,d)$ and $\sMbar_{1,n}^m(\PP^r,d)$, and in particular deduce that $\sMbar_{1,n}^m(\PP^r,d)$ is irreducible for $m \geq \min(r,d) + n$.  We show that $\sMbar_{1,n}^m(\PP^r,d)$ is smooth if $d+n \leq m \leq 5$.
\end{abstract}

\maketitle

\def\thepart{\Roman{part}}

\tableofcontents

\section{Introduction}

\subsection{Background}
Let $\sM_{g,n}$ be the moduli stack of smooth, genus $g$, $n$-pointed curves over an algebraically closed field $k$.  In \cite{DM}, Deligne and Mumford compactified this space by introducing the notion of a stable curve:
\begin{definition}
Let $C$ be a reduced, connected, proper curve of arithmetic genus $g$, and let $p_1, \ldots, p_n \in C$ be smooth, distinct points.  Then $(C, p_1, \ldots, p_n)$ is said to be \emph{stable} if the following conditions hold:
\begin{enumerate}
	\item $C$ has only nodes as singularities.
	\item $C$ has no non-trivial infinitesimal automorphisms identical on $p_1, \ldots, p_n$.
\end{enumerate}
\end{definition}
Deligne and Mumford proved that the stack $\sMbar_{g,n}$ of stable curves, in addition to being proper, is irreducible and smooth over $\Spec \ZZ$.

Motivated by predictions from mirror symmetry on the enumerative geometry of rational curves, Kontsevich extended the definition of a stable curve to the setting of maps into a scheme $X$.  Kontsevich introduced the notion of a stable map \cite{Ko}:
\begin{definition}
Let $C$ be a reduced, connected, proper curve of arithmetic genus $g$, and let $p_1, \ldots, p_n \in C$ be smooth, distinct points.  A map $\mu: C \ra X$ is said to be \emph{stable} if the following conditions hold:
\begin{enumerate}
	\item $C$ has only nodes as singularities.
	\item $\mu$ has no non-trivial infinitesimal automorphisms identical on $p_1, \ldots, p_n$ and $X$.
\end{enumerate}
\end{definition}
For simplicity, we assume that $k = \CC$ whenever discussing stable maps in the remainder of the introduction; see Section \ref{prelim: m-stable maps} for the setup in general.  For any homology class $\beta \in H_2(X,\ZZ)$, let
\[
\sMbar_{g,n}(X,\beta)
\]
\noindent be the stack of stable maps $\mu: C \ra X$ from curves of arithmetic genus $g$ with $n$ marked points to $X$ such that $\mu_*([C]) = \beta$.  For $X$ projective, $\sMbar_{g,n}(X,\beta)$ is a proper Deligne-Mumford stack of finite type.  For $g=0$ and $X$ a smooth convex variety (a class of varieties which includes all homogeneous spaces), $\sMbar_{0,n}(X,\beta)$ is irreducible and smooth.  However, in general $\sMbar_{g,n}(X,\beta)$ is not well-behaved: even when $X$ is projective space, $\sMbar_{g,n}(X,\beta)$ can have multiple singular irreducible components of varying dimensions if $g > 0$.

Recent work of Hassett and others on the minimal model program for $\sMbar_{g,n}$ has led to interest in alternate compactifications of $\sM_{g,n}$.  One goal of this program is to describe the log canonical models of $\sMbar_{g,n}$, defined by
\[
\Mbar_{g,n}(\alpha) := \Proj \oplus_{m \geq 0} H^0( \sMbar_{g,n}, m(K_{\sMbar_{g,n}} + \alpha \delta) ),
\]
\noindent where $\delta$ is the boundary divisor on $\sMbar_{g,n}$ and $0 \leq \alpha \leq 1$ is rational.  A remarkable feature of these spaces is that they are often modular:  that is, they coarsely represent functors defined by moduli problems.  For instance, $\Mbar_{g,0}(9/11)$ is isomorphic to $\sMbar_{g,0}^{ps}$, the space of pseudo-stable curves constructed by Schubert via geometric invariant theory \cite{Sch, HH}.  In the case $g=1$, Smyth has constructed a series of compactifications $\sMbar_{1,n}(m)$ of $\sM_{1,n}$ for every integer $m \geq 1$ using a generalized notion of stability \cite{S1}:
\begin{definition}
Let $C$ be a reduced, connected, proper curve of arithmetic genus one, and let $p_1, \ldots, p_n \in C$ be smooth, distinct points.  Then $(C, p_1, \ldots, p_n)$ is said to be \emph{$m$-stable} if the following conditions hold:
\begin{enumerate}
	\item $C$ has only nodes and elliptic $l$-fold points, $l \leq m$, as singularities (see Definition \ref{intro: elliptic} below).
	\item For any connected subcurve $E \subset C$ of arithmetic genus one,
	\[
	\abs{ \{E \cap \ol{C \backslash E} \} \cup \{p_i: p_i \in E\} } > m.
	\]
	\item $C$ has no non-trivial infinitesimal automorphisms identical on $p_1, \ldots, p_n$.
\end{enumerate}
\end{definition}
The spaces $\sMbar_{1,n}(m)$ are proper, irreducible, Deligne-Mumford stacks over $\Spec \ZZ[1/6]$ (see Remark \ref{remark: extra automorphisms} for an explanation of the restriction $\ch k \neq 2,3$).  After further allowing the marked points $p_i$ to collide suitably \cite{S1, H}, Smyth has shown that every log canonical model $\sMbar_{1,n}(\alpha)$ is isomorphic to the normalization of the coarse moduli space of one of the above compactifications \cite{S2}.

\subsection{Statement of the main results}
It is natural to consider analogues of the above alternate compactifications in the setting of maps to a projective variety $X$.  For any homology class $\beta \in H_2(X, \ZZ)$, let $\sM_{1,n}(X,\beta)$ be the stack of maps from smooth curves of genus one with $n$ marked points to $X$ such that $\mu_*([C]) = \beta$.  In this paper, we construct a series of alternate compactifications of $\sM_{1,n}(X,\beta)$ by introducing the notion of an $m$-stable map:
\begin{definition}
Let $C$ be a reduced, connected, proper curve of arithmetic genus one, and let $p_1, \ldots, p_n \in C$ be smooth, distinct points.  A map $\mu: C \ra X$ is said to be \emph{$m$-stable} if the following conditions hold:
\begin{enumerate}
	\item $C$ has only nodes and elliptic $l$-fold points, $l \leq m$, as singularities.
	\item For any connected subcurve $E \subset C$ of arithmetic genus one on which $\mu$ is constant,
	\[
	\abs{ \{E \cap \ol{C \backslash E} \} \cup \{p_i: p_i \in E\} } > m.
	\]
	\item $\mu$ has no non-trivial infinitesimal automorphisms identical on $p_1, \ldots, p_n$ and $X$.
\end{enumerate}
\end{definition}
The definition of an $m$-stable map of class $\beta$ extends to a moduli functor in the usual way, and the main result of Part I of this paper (Theorem \ref{theorem: construction}) is:
\begin{theorem}
The functor of $m$-stable maps of class $\beta$ is representable by a proper Deligne-Mumford stack $\sMbar_{1,n}^m(X,\beta)$ of finite type.
\end{theorem}
In fact, we prove this over any algebraically closed field $k$ with $\ch k \neq 2,3$, but need an additional condition in the definition of $m$-stability (condition (4) in Definition \ref{def: m-stable map}); the same issue arises with Kontsevich's spaces of stable maps for $\ch k > 0$.  The proof of Theorem \ref{theorem: construction} is fairly standard except for the valuative criterion for properness, which we show using a modification of the method in \cite{S1}.  Given a 1-parameter family of $m$-stable curves with smooth generic fiber, Smyth considers the semi-stable reduction of the family, then performs a prescribed sequence of blow-ups and contractions on the central fiber until it eventually becomes $m$-stable.  In order to apply an analogue of semi-stable reduction in our setting, we must first reduce to the case of families with smooth domain in the generic fiber; this extra step is necessary because $\sMbar_{1,n}^m(X,\beta)$ is reducible in general.  This step is also slightly more subtle than in the Kontsevich-stable case (as carried out in, e.g., \cite{FP}) since it is less clear how to normalize families of curves with elliptic $m$-fold points.  We then apply Smyth's sequence of blow-ups and contractions on the domain of the central fiber; however, since we cannot contract any component on which the map is non-constant, we potentially have to terminate the process earlier than in \cite{S1}.

In Part II, we study the irreducible components and smoothness of $\sMbar_{1,n}^m(X,\beta)$ in the case $X = \PP^r$.  As mentioned above, $\sMbar_{1,n}(X,\beta)$ can have multiple irreducible components even for $X = \PP^r$ and $\beta = d[\text{line}]$ (abbreviated $\beta = d$).  In the case $(X,\beta) = (\PP^r,d)$, these can be divided into a principal component whose generic element is a map with smooth domain (this locus is reducible in general for $g > 1$) and a number of ``extra'' components.  In Section \ref{section: irreducible components}, we show that the spaces of $m$-stable maps gradually ``cut away'' the extra components of $\sMbar_{1,n}(\PP^r,d)$ as $m$ increases, and eventually become irreducible:
\begin{theorem}
For $m \geq \min(r,d) + n$, $\sMbar_{1,n}^m(\PP^r,d)$ is irreducible.
\end{theorem}
We actually obtain a much more precise result:  we give explicit descriptions of all of the irreducible components of the spaces $\sMbar_{1,n}(\PP^r,d)$ and $\sMbar_{1,n}^m(\PP^r,d)$ in Theorem \ref{theorem: irreducible components, n>0}.  The main result which makes this possible is an explicit description of the principal component of $\sMbar_{1,n}^m(\PP^r,d)$.  Note that such a description was stated in the case of genus one Kontsevich-stable maps in \cite{VZ1}.

Given that $\sMbar_{1,n}^m(\PP^r,d)$ is irreducible for $m \geq \min(r,d) + n$, we naturally ask if it is smooth.  In most cases the answer is no, but in Section \ref{section: smoothness} we show that the answer is yes if $d+n \leq m \leq 5$.  In order to describe this in more detail, we first go back and give the definition of an elliptic $m$-fold point introduced by Smyth:
\begin{definition}
\label{intro: elliptic}
We say that $p$ is an \emph{elliptic $m$-fold point} of a curve $C$ if
\[
\hat{\sO}_{C,p} \cong
\begin{cases}
k[[x,y]]/(y^2-x^3) 							& m=1 			\ \ \text{(ordinary cusp)} 				\\
k[[x,y]]/(y^2-x^2 y) 						& m=2 			\ \ \text{(ordinary tacnode)}			\\
k[[x,y]]/(x^2 y-x y^2) 					& m=3				\ \ \text{(planar triple point)}  \\
k[[x_1, \ldots, x_{m-1}]]/I_m		& m \geq 4, \, \text{($m$ general lines through the origin in $\bA^{m-1}$)},
\end{cases}
\]
\[
\hspace{-75pt} I_m := (x_h x_i - x_h x_j: i,j,h \in \{ 1, \ldots, m-1 \} \text{ distinct}).
\]
\end{definition}
We will need a result on the deformation theory of the elliptic $m$-fold point proved by Smyth in \cite{S2}: namely, that deformations of the elliptic $m$-fold point are unobstructed if and only if $m \leq 5$.  The local structure of $\sMbar_{1,n}(m)$ at a point $[C]$ is governed by the miniversal deformation space of $C$, and the conclusion is that $\sMbar_{1,n}(m)$ is smooth if and only if $m \leq 5$.  In the setting of $m$-stable maps, the key fact is that for $m \geq d+n$, an $m$-stable map cannot be constant on any genus one subcurve of the domain (Lemma \ref{lemma: stabilization}).  A simple cohomological argument then gives that deformations of $m$-stable maps to $\PP^r$ are trivially unobstructed for $m \leq 5$, and the conclusion (Theorem \ref{theorem: smoothness}) is:
\begin{theorem}
The stack $\sMbar_{1,n}^m(\PP^r,d)$ is smooth if $d+n \leq m \leq 5$.
\end{theorem}

Of course, these are rather restrictive conditions for smoothness.  However, we note that Smyth has proven in \cite{S2} that the normalizations of the coarse moduli spaces of $n$-pointed $m$-stable curves are projective and $\QQ$-factorial, and is therefore able to carry out intersection theory with $\QQ$-coefficients on these normalized coarse moduli spaces.  We expect to prove analogous results for our spaces in future work.

\subsection{Outline of the paper}
In Section \ref{section: prelim} we summarize some necessary facts from \cite{S1} and reformulate the condition that a map have no non-trivial infinitesimal automorphisms in terms of distinguished points on irreducible components of the domain curve.  We then define $m$-stable maps of class $\beta$ over any algebraically closed field $k$ with $\ch k \neq 2, 3$.  In Section \ref{section: construction} we prove that the space of $m$-stable maps $\sMbar_{1,n}^m(X,\beta)$ is a Deligne-Mumford stack and prove the valuative criterion for properness.  After this point in the paper, we only consider the case $(X,\beta) = (\PP^r,d)$.  In Section \ref{section: stabilization} we observe that $\sMbar_{1,n}^m(\PP^r,d)$ stabilizes for $m \geq d+n$.  In Section \ref{section: irreducible components} we explicitly describe the irreducible components of $\sMbar_{1,n}(\PP^r,d)$ and $\sMbar_{1,n}^m(\PP^r,d)$, and obtain as a corollary that $\sMbar_{1,n}^m(X,\beta)$ is irreducible for $m \geq \min(r,d)+n$.  In Section \ref{section: smoothness} we prove that the limiting space $\sMbar_{1,n}^{d+n}(\PP^r,d)$ is smooth if $d+n \leq 5$.  Finally, in Section \ref{section: plane cubics} we illustrate our results in the example of spaces of plane cubics $\sMbar_{1,0}^m(\PP^2,3)$.  The Kontsevich space is well-understood in this example, and we give a concrete description of how this space ``evolves'' as $m$ increases.

\subsection{Future directions and related work}

\subsubsection{Higher genera}
As described in the introduction to \cite{S1}, the construction of $\sMbar_{1,n}(m)$ does not directly extend to higher genera.  See \cite{S3} for an approach in all genera which allows a much wider array of singularities.  We simply note here that in the case $m=1$, the space $\sMbar_{g,n}^1(X,\beta)$ with $\beta \neq 0$ should exist in all genera including $g=2$: the example causing non-separatedness in the case of curves is no longer an issue because the map must be non-constant on at least one of the genus one subcurves.  It would be interesting to see if the resulting Gromov-Witten theory yields new information (see the next remark).

\subsubsection{Intersection theory and enumerative geometry}
In future work, we expect to show that the space $\sMbar_{1,n}^m(X,\beta)$ always carries a virtual fundamental class.  We expect that the resulting primary Gromov-Witten invariants will agree with the classical ones, but that the gravitational descendants will yield new information (i.e.\ will not be expressible in terms of the classical Gromov-Witten invariants of $X$).

\subsubsection{Projectivity of the coarse moduli space}
It is known that the coarse moduli spaces of $m$-stable curves are projective \cite{S2}.  By general theory, the coarse moduli spaces of $m$-stable maps exist as proper algebraic spaces \cite{A}.  In future work, we expect to be able to use the methods of \cite{FP} and \cite{K1} to show projectivity of these spaces.

\subsubsection{Other compactifications of $\sM_{g,n}(X,\beta)$}
Recently there have been a remarkable number of new compactifications of $\sM_{g,n}(X,\beta)$.  We outline five of these constructions here and compare them, in rather rough terms, to the one given in this paper.  All of the constructions below exist as proper Deligne-Mumford stacks.  \\

\noindent (1) \emph{Weighted stable maps: $\sMbar_{g,\sA}(X,\beta)$} (Alexeev and Guy \cite{AG}/Bayer and Manin \cite{BaM}/\Mustata \ and \Mustata \ \cite{MM}):  \\

These papers extend Hassett's weighted stability condition \cite{H} to the setting of maps.  It is certainly possible to merge these results with ours to obtain spaces $\sMbar_{1,\sA}^m(X,\beta)$ of weighted $m$-stable maps.  Given the results of \cite{S2}, we expect the normalized principal components of the corresponding coarse moduli spaces to be log canonical models of the principal component of $\sMbar_{1,n}(X,\beta)$.  \\

\noindent (2) \emph{Desingularization of the principal component of $\sMbar_{1,n}(\PP^r,d)$: $\widetilde{\sM}_{1,n}^0(\PP^r,d)$} (Vakil and Zinger \cite{VZ1, VZ2}):  \\

These papers give an explicit desingularization of the principal component of $\sMbar_{1,n}(\PP^r,d)$.  This is achieved through a sequence of blow-ups along the intersection of the principal component and collections of the remaining irreducible components, and has the same effect as our spaces of cutting away irreducible components.  Given the results of \cite{S2}, we expect that our construction is instead achieved by a sequence of birational contractions along the same intersections.  The spaces $\widetilde{\sM}_{1,n}(\PP^r,d)$ are not known to have modular interpretations, although they have many of the properties needed for applications to Gromov-Witten theory.  \\

\noindent (3) \emph{Stable ramified maps: $\sMbar_{g,n,\mu}(X,\beta)$} (Kim, Kresch, and Oh \cite{KKO}):  \\

This paper considers a space of maps from smooth curves to $X$ with prescribed ramification indices at the marked points.  As the marked points approach one another, the map degenerates to a map from a nodal curve to a fiber of the ``universal family'' over a Fulton-MacPherson configuration space of distinct ordered points on $X$.  It does not seem to be known how to relate this to the other compactifications.  \\

\noindent (4) \emph{Logarithmic stable maps: $\Mbar_{g,n}^{\text{log}}(\sU/\sB)$} (Kim \cite{Ki}):  \\

This paper defines stable maps in the setting of logarithmic geometry.  We simply point out one byproduct:  the construction of a smooth modular compactification of $\sM_{1,n}(\PP^r,d)$.  One notable feature of this stack is that only parametrizes (log) maps which are non-constant on any genus one subcurve of the domain.  It is not known how this space relates to Vakil and Zinger's desingularization.  \\

\noindent (5) \emph{Stable quotients: $\ol{Q}_1(\PP^r,d)$} (Marian, Oprea, and Pandharipande \cite{MOP}):  \\

This paper defines a space of stable quotients of a rank $n$ trivial sheaf on stable curves.  One byproduct is the construction of a smooth modular compactification of $\sM_{1,0}(\PP^r,d)$.  It may be interesting to relate this to our spaces and Kim's spaces.  \\

\subsection{Notation}
All curves and varieties are defined over an algebraically closed field $k$.  A \emph{curve} is a reduced connected one-dimensional scheme of finite type over $k$.  The \emph{genus} of a curve is always its arithmetic genus.  An \emph{$n$-pointed curve} is a curve $C$ together with $n$ distinct smooth ordered points $p_1, \ldots, p_n \in C$.  A curve is \emph{Gorenstein} if its dualizing sheaf is invertible.  A \emph{projective variety} $X$ is always equipped with a given very ample line bundle $\sO_X(1)$.  A \emph{map}, unless otherwise noted, is a regular morphism of varieties or schemes.  For any irreducible component $F$ of a pointed curve $(C, p_1, \ldots, p_n)$, we say that a point on $F$ is \emph{distinguished} if it is either a singular or marked point of $C$, i.e.\ the set of distinguished points of $F$ is
\[
\{F \cap \ol{C \backslash F} \} \cup \{p_i: p_i \in F\}.
\]
$\Delta$ always denotes an open disc when $k = \CC$, or $\Spec(A)$ with $A$ a discrete valuation ring with residue field $k$ in general; 0 and $\eta$ denote the closed and generic points of $\Delta$.  \\

\noindent \textbf{Acknowledgments.}
I would like to thank my adviser, Prof.\ Joe Harris, for his invaluable support, guidance, discussions, and humor throughout the course of this work, and David Smyth and Frederick van der Wyck for many helpful discussions.

\newpage

\part{Moduli spaces of $m$-stable maps}

\section{Preliminaries}
\label{section: prelim}

In Section \ref{prelim: m-stable curves}, we summarize some facts from \cite{S1} that we will need.  In Section \ref{prelim: m-stable maps}, we reformulate the condition that a map have no non-trivial infinitesimal automorphisms in terms of distinguished points (Lemma \ref{lemma: map automorphisms}), define $m$-stable maps (Definition \ref{def: m-stable map}), and define $m$-stable maps of class $\beta$ (Definition \ref{def: class beta}).

\subsection{$m$-stable curves}
\label{prelim: m-stable curves}

\subsubsection{Elliptic $m$-fold points}
\label{prelim: m}
In \cite{S1}, Smyth introduced the following sequence of curve singularities:
\begin{definition}
\label{def: elliptic}
We say that $p$ is an \emph{elliptic $m$-fold point} of a curve $C$ if
\[
\hat{\sO}_{C,p} \cong
\begin{cases}
k[[x,y]]/(y^2-x^3) 							& m=1 			\ \ \text{(ordinary cusp)} 				\\
k[[x,y]]/(y^2-x^2 y) 						& m=2 			\ \ \text{(ordinary tacnode)}			\\
k[[x,y]]/(x^2 y-x y^2) 					& m=3				\ \ \text{(planar triple point)}  \\
k[[x_1, \ldots, x_{m-1}]]/I_m		& m \geq 4, \, \text{($m$ general lines through the origin in $\bA^{m-1}$)},
\end{cases}
\]
\[
\hspace{-75pt} I_m := (x_h x_i - x_h x_j: i,j,h \in \{ 1, \ldots, m-1 \} \text{ distinct}).
\]
\end{definition}

A key fact about these singularities is that they are \emph{Gorenstein}; that is, any curve containing only these singularities has invertible dualizing sheaf.  More specifically, we have the following result:
\begin{lemma}[\cite{S1}, Proposition 2.5]
\label{lemma: Gorenstein}
Let $C$ be a curve with an elliptic $m$-fold point $p \in C$, and let $\tilde{C} \ra C$ be the normalization of $C$ at $p$.  Then
\begin{enumerate}
	\item $\omega_C$ is invertible near $p$, i.e.\ the elliptic $m$-fold point is Gorenstein.
	\item $\pi^* \omega_C = \omega_{\tilde{C}}(2p_1 + \ldots + 2p_m)$.
\end{enumerate}
\end{lemma}
A second local property of these singularities which we will need is that they are \emph{smoothable}:
\begin{lemma}[\cite{S1}, proof of Proposition 2.11]
\label{lemma: elliptic m-fold point is smoothable}
The elliptic $m$-fold point is smoothable.
\end{lemma}
Elliptic $m$-fold points arise in the framework of birational geometry.  The next lemma states that given any 1-parameter family of genus $g$ curves with smooth generic fiber, contraction of a genus one subcurve of the central fiber via a suitable line bundle yields an elliptic $m$-fold point in the central fiber:

\begin{lemma}[\cite{S1}, Lemma 2.12, Gorenstein case]
Let $\pi: \sC \ra \Delta$ be a flat family of Gorenstein curves of genus $g$ with central fiber $C$ and smooth generic fiber.  Let $\sL$ be a line bundle on $\sC$ with positive degree on the generic fiber and non-negative degree on each component of the special fiber.  Assume that
\begin{enumerate}
	\item $E := \{ \text{Irreducible components } F \subset C: \deg \sL|_F = 0 \}$ is connected of arithmetic genus one,
	\item $\sL|_E \cong \sO_E$,
	\item Each point $p \in E \cap \ol{C \backslash E}$ is a node of $C$,
	\item Each point $p \in E \cap \ol{C \backslash E}$ is a regular point of $\sC$.
\end{enumerate}
Then $\sL$ is $\pi$-semiample and defines a contraction $\phi$
\begin{diagram}
\sC &						&\rTo^\phi& 							&	\sC'  \\
		& \rdTo_\pi	&		 			& \ldTo_{\pi'}	&				\\
		&						& \Delta	&								&
\end{diagram}
where $\phi$ is proper and birational with exceptional locus $E$.  Furthermore, if $\sL$ is of the form
\[
\omega_{\sC/\Delta}(D+\Sigma)
\]
with $D$ a Cartier divisor supported on $E$ and $\Sigma$ a Cartier divisor disjoint from $E$, then $p := \phi(E)$ is an elliptic $m$-fold point with
\[
m = \abs{E \cap \ol{C \backslash E} }.
\]
\label{lemma: contraction}
\end{lemma}

\subsubsection{Genus one curves}
\label{prelim: genus one}
We now turn our attention to global aspects of the geometry of genus one curves.  We will constantly use the following fact about the topology of a Gorenstein curve of arithmetic genus one:

\begin{lemma}[\cite{S1}, Lemma 3.1]
Let $C$ be a Gorenstein curve of arithmetic genus one.  Then $C$ contains a unique subcurve $Z \subset C$ satisfying
\begin{enumerate}
	\item $Z$ is connected,
	\item $Z$ has arithmetic genus one,
	\item $Z$ has no disconnecting nodes.
\end{enumerate}
We call $Z$ the \emph{minimal elliptic subcurve of $C$}.  We write
\[
C = Z \cup R_1 \cup \ldots \cup R_k,
\]
where $R_1, \ldots, R_k$ are the connected components of $\ol{C \backslash Z}$, and call this \emph{the fundamental decomposition of $C$}.  Each $R_i$ is a nodal curve of arithmetic genus zero meeting $Z$ in a single point, and $Z \cap R_i$ is a node of $C$.
\label{lemma: fundamental}
\end{lemma}

We can also classify all possible minimal elliptic subcurves appearing in Lemma \ref{lemma: fundamental}:

\begin{lemma}[\cite{S1}, Lemma 3.3]
Let $Z$ be a Gorenstein curve of arithmetic genus one with no disconnecting nodes.  Then $Z$ is one of the following:
\begin{enumerate}
	\item A smooth elliptic curve,
	\item An irreducible rational nodal curve,
	\item A ring of $\PP^1$'s, or
	\item $Z$ has an elliptic $m$-fold point $p$ and the normalization of $Z$ at $p$ consists of $m$ distinct, smooth rational curves.
\end{enumerate}
Moreover, in all four cases, $\omega_Z \cong \sO_Z$.
\label{lemma: minimal}
\end{lemma}
Lemmas \ref{lemma: elliptic m-fold point is smoothable}, \ref{lemma: fundamental}, and \ref{lemma: minimal} give the following fact which will be useful in Section \ref{subsection: smoothability}:
\begin{lemma}
\label{lemma: Gorenstein is smoothable}
Any Gorenstein curve of arithmetic genus one is smoothable.
\end{lemma}
The following lemma will allow us to phrase the condition of $m$-stability without making reference to the minimal elliptic subcurve:
\begin{lemma}[\cite{S1}, Corollary 3.6]
Let $(C, p_1, \ldots, p_n)$ be a pointed Gorenstein curve of arithmetic genus one, and suppose that every smooth rational component of $C$ has at least two distinguished points.  Let $Z \subset C$ be the minimal elliptic subcurve.  Then
\[
\abs{ \{Z \cap \ol{C \backslash Z} \} \cup \{p_i: p_i \in Z\} } > m.
\]
\noindent if and only if
\[
\abs{ \{E \cap \ol{C \backslash E} \} \cup \{p_i: p_i \in E\} } > m
\]
for every connected arithmetic genus one subcurve $E \subset C$.
\label{lemma: any E}
\end{lemma}

Finally, the following definition will be convenient in the proof of the valuative criterion for properness:

\begin{definition}[\cite{S1}, Definition 3.4]
\label{def: level}
Let $(C, p_1, \ldots, p_n)$ be a pointed Gorenstein curve of arithmetic genus one, and let $Z \subset C$ be the minimal elliptic subcurve.  Then the \emph{level} of $(C, p_1, \ldots, p_n)$ is defined to be the number of distinguished points of $Z$, i.e.
\[
\abs{ \{Z \cap \ol{C \backslash Z} \} \cup \{p_i: p_i \in Z\} }.
\]
\end{definition}

\subsubsection{Infinitesimal automorphisms}
By definition, an infinitesimal automorphism of a pointed curve $(C, p_1, \ldots, p_n)$ is an element of the tangent space $T_{[\id]} \Aut(C, p_1, \ldots, p_n)$.

\begin{lemma}
Assume $\ch k \neq 2, 3$.  Let $(C, p_1, \ldots, p_n)$ be a pointed Gorenstein genus one curve.  Then
\[
\InfAut(C, p_1, \ldots, p_n) = 0
\]
\noindent if and only if the following conditions hold:
\begin{enumerate}
		\item If $C$ is nodal, then every rational component of $C$ has at least three distinguished points.
		\item If $C$ has a (unique) elliptic $m$-fold point $p$, and $B_1, \ldots, B_m$ are the irreducible components of the minimal elliptic subcurve $Z \subset C$, then
		\begin{enumerate}
		\renewcommand{\labelenumiii}{(b\arabic{enumiii})}
			\item Each $B_i$ has at least 2 distinguished points.
			\item Some $B_i$ has at least 3 distinguished points.
			\item Every other component of $C$ has at least 3 distinguished points.
		\end{enumerate}
	\end{enumerate}
\label{lemma: curve automorphisms}
\end{lemma}
\begin{proof}
This is Corollary 2.4 in \cite{S1} (in the case that $C$ is Gorenstein) combined with the well-known isomorphism
\[
\InfAut(C, p_1, \ldots, p_n) \cong H^0(T_C(-p_1- \ldots-p_n)).
\]
\end{proof}

\begin{remark}
\label{remark: extra automorphisms}
The proof of Lemma \ref{lemma: curve automorphisms} relies on the assumption that regular vector fields on $C$ are always induced by regular vector fields on the normalization $\tilde{C}$.  This assumption is true in characteristic 0, but not necessarily in positive characteristic; see Section 2.1 of \cite{S1} for an example.
\end{remark}

We can now give the definition of an $m$-stable curve.  For simplicity, we do not incorporate the weighted variant defined in \cite{S1} in this paper (but as noted in the introduction, Hassett's weighted stability condition has been extended to the setting of maps).
\begin{definition}
Assume $\ch k \neq 2, 3$.  Fix any integers $0 \leq m < n$.  A pointed genus one curve $(C, p_1, \ldots, p_n)$ is said to be \emph{$m$-stable} if the following conditions hold:
\begin{enumerate}
	\item $C$ has only nodes and elliptic $l$-fold points, $l \leq m$, as singularities.
	\item The number of distinguished points on the minimal elliptic subcurve $Z \subset C$ is strictly greater than $m$.  Equivalently, by Lemma \ref{lemma: any E},
	\[
	\abs{ \{E \cap \ol{C \backslash E} \} \cup \{p_i: p_i \in E\} } > m
	\]
	for any connected subcurve $E \subset C$ of arithmetic genus one.
	\item $C$ has no non-trivial infinitesimal automorphisms identical on $p_1, \ldots, p_n$.  Equivalently, by Lemma \ref{lemma: curve automorphisms},
	\begin{enumerate}
		\item If $C$ is nodal, then every rational component of $C$ has at least three distinguished points.
		\item If $C$ has a (unique) elliptic $m$-fold point $p$, and $B_1, \ldots, B_m$ are the irreducible components of the minimal elliptic subcurve $Z \subset C$, then
		\begin{enumerate}
		\renewcommand{\labelenumiii}{(b\arabic{enumiii})}
			\item Each $B_i$ has at least 2 distinguished points.
			\item Some $B_i$ has at least 3 distinguished points.
			\item Every other component of $C$ has at least 3 distinguished points.
		\end{enumerate}
	\end{enumerate}
\end{enumerate}
\label{def: m-stable curve}
\end{definition}

\begin{remark}
For $m=0$, this is the classical definition of a Deligne-Mumford stable genus one curve.  The construction of the space of $m$-stable curves in \cite{S1} is given for $m \geq 1$ since it uses classical semi-stable reduction (i.e.\ properness of $\sMbar_{1,n}$).
\label{remark: m=0 curves}
\end{remark}

\subsection{$m$-stable maps}
\label{prelim: m-stable maps}

Let $(C, p_1, \ldots, p_n)$ be a pointed curve, and let $\mu: C \ra X$ be a map to a projective variety $X$.  An \emph{automorphism} of $\mu$ is an automorphism $\phi$ of $C$ such that $\phi(p_i) = p_i$ and $\mu \circ \phi = \mu$.  We denote the group of automorphisms of $\mu$ by $\Aut(C, p_1, \ldots, p_n, \mu)$ or simply $\Aut(\mu)$.  An \emph{infinitesimal automorphism} of $\mu$ is an element of the tangent space $T_{[\id]} \Aut(\mu)$.  We denote the vector space of infinitesimal automorphisms of $\mu$ by $\InfAut(C, p_1, \ldots, p_n, \mu)$ or $\InfAut(\mu)$.

We want to translate the condition $\InfAut(\mu) = 0$ into a condition on distinguished points on the domain curve.  We first adopt the following condition to ensure separability of $\mu$, well-known from \cite{BM}:  if $\ch k > 0$, we say that $\mu$ is \emph{bounded by the characteristic} if
\[
\deg \mu^*(\sO_X(1)) < \ch k.
\]
If $\ch k = 0$, we say that every map $\mu: C \ra X$ is bounded by the characteristic.  Note that every map bounded by the characteristic is separable.

\begin{lemma}
\label{lemma: map automorphisms}
Assume $\ch k \neq 2, 3$.  Let $(C, p_1, \ldots, p_n)$ be a pointed Gorenstein genus one curve, and let $\mu: C \ra X$ be a map to a projective variety $X$.  Assume that $\mu$ is bounded by the characteristic.  Then
\[
\InfAut(C, p_1, \ldots, p_n, \mu) = 0
\]
\noindent if and only if the following conditions hold:
\begin{enumerate}
		\item If $C$ is nodal, then every rational component of $C$ on which $\mu$ is constant has at least three distinguished points.
		\item If $C$ has a (unique) elliptic $m$-fold point $p$, and $B_1, \ldots, B_m$ are the irreducible components of the minimal elliptic subcurve $Z \subset C$, then
		\begin{enumerate}
		\renewcommand{\labelenumiii}{(b\arabic{enumiii})}
			\item If $\mu$ is constant on $B_i$, then $B_i$ has at least 2 distinguished points.
			\item If $\mu$ is constant on $Z$, then some $B_i$ has at least 3 distinguished points.
			\item Every other component of $C$ on which $\mu$ is constant has at least 3 distinguished points.
		\end{enumerate}
	\end{enumerate}
\end{lemma}
\begin{proof}
If $C$ is nodal, it classically known that $\InfAut(\mu) = 0$ if and only if (1) holds.  Suppose (2) holds.  We may assume without loss of generality that $X = \PP^r$.  Since $\mu$ is separable, we can choose hyperplanes $H_1, H_2 \subset \PP^r$ which intersect the map $\mu$ transversely at unmarked smooth points, i.e.\ the divisors $\mu^*(H_1), \mu^*(H_2)$ consist of reduced unmarked smooth points on $C$.  Write $\mu^*(H_j) = \sum q_{j,k}$, $1 \leq k \leq d$, where $d$ is the degree of $\mu$.  Then the pointed curve
\[
(C, \{ p_i, q_{j,k} \})
\]
\noindent satisfies the conditions of Lemma \ref{lemma: curve automorphisms}, so $\InfAut(C, \{ p_i, q_{j,k} \}) = 0$.  Any automorphism of $(C, \{ p_i \})$ which commutes with $\mu$ must be the identity on $\mu^*(H_j) = \sum q_{j,k}$, so we have an injection
\[
\Aut(\mu) \hra \Aut(C, \{ p_i, q_{j,k} \}),
\]
\noindent hence an injection
\[
\InfAut(\mu) \hra \InfAut(C, \{ p_i, q_{j,k} \}),
\]
\noindent so $\InfAut(\mu) = 0$.

Suppose $\InfAut(\mu) = 0$.  2(c) is clear since any such component intersects the rest of the curve nodally.  We now directly construct automorphisms of $Z$ to establish conditions 2(a) and 2(b).

If 2(a) does not hold, then $\mu$ is constant on some $B_i$ with no distinguished points other than $p$.  Given $\phi_j \in \Aut(B_j)$ for every $j$ which are the identity at $p$ and which agree to order 2 at $p$, it is clear that the $\phi_j$ glue together to an automorphism $\phi \in \Aut(Z)$.  There is a non-trivial 1-parameter family of $\phi_i \in \Aut(B_i)$ equal to the identity to order 2 at $p$ (i.e.\ preserving the subscheme $2p$).  Gluing this family with $\id \in \Aut(B_j)$ for all $j \neq i$ yields a 1-parameter family of non-trivial automorphisms of $Z$ (hence of $C$).  These automorphisms commute with $\mu$, so we conclude that $\InfAut(\mu) \neq 0$.

Similarly, if 2(b) does not hold, then $\mu$ is constant on $Z$ and each $B_i$ has at most one distinguished point $q_i$ other than the elliptic $m$-fold point $p$.  For each $i$, choose an isomorphism $B_i \cong \PP^1$ so that $q_i = \infty$; then multiplication by any non-zero $\lambda \in \CC$ gives a 1-parameter family of automorphisms of $Z$ which commute with $\mu$, so again $\InfAut(\mu) \neq 0$.
\end{proof}

We can now give the definition of an $m$-stable map:

\begin{definition}
Assume $\ch k \neq 2, 3$.  Fix any integers $m \geq 0$ and $n \geq 0$.  Let $(C, p_1, \ldots, p_n)$ be a pointed genus one curve, and let $\mu: C \ra X$ be a map to a projective variety $X$.  Then $(C, p_1, \ldots, p_n, \mu)$, or simply $\mu$, is said to be \emph{$m$-stable} if the following conditions hold:
\begin{enumerate}
	\item $C$ has only nodes and elliptic $l$-fold points, $l \leq m$, as singularities.
	\item If $\mu$ is constant on the minimal elliptic subcurve $Z \subset C$, then 
	\[
	\abs{ \{ Z \cap \ol{C \backslash Z} \} \cup \{p_i: p_i \in Z \} } > m.
	\]
	Equivalently, by Lemma \ref{lemma: any E},
	\[
	\abs{ \{ E \cap \ol{C \backslash E} \} \cup \{p_i: p_i \in E \} } > m
	\]
	for any connected arithmetic genus one subcurve $E \subset C$ on which $\mu$ is constant.
	\item $\mu$ has no non-trivial infinitesimal automorphisms identical on $p_1, \ldots, p_n$ and $X$.  Equivalently, by Lemma \ref{lemma: map automorphisms} and condition (4),
	\begin{enumerate}
		\item If $C$ is nodal, then every rational component of $C$ on which $\mu$ is constant has at least three distinguished points.
		\item If $C$ has a (unique) elliptic $m$-fold point $p$, and $B_1, \ldots, B_m$ are the irreducible components of the minimal elliptic subcurve $Z \subset C$, then
		\begin{enumerate}
		\renewcommand{\labelenumiii}{(b\arabic{enumiii})}
			\item If $\mu$ is constant on $B_i$, then $B_i$ has at least 2 distinguished points.
			\item If $\mu$ is constant on $Z$, then some $B_i$ has at least 3 distinguished points.
			\item Every other component of $C$ on which $\mu$ is constant has at least 3 distinguished points.
		\end{enumerate}
	\end{enumerate}
	\item $\mu$ is bounded by the characteristic.
\end{enumerate}
\label{def: m-stable map}
\end{definition}

\begin{remark}
Unlike $m$-stability for curves, we do not need $m < n$ here since condition (2) automatically holds whenever $\mu$ is non-constant on the minimal elliptic subcurve $Z \subset C$.
\end{remark}

\begin{remark}
For $m=0$, this is the definition of a stable map given by Kontsevich.  For $X$ a point, this is the definition of an $m$-stable curve given by Smyth.  In parallel with Remark \ref{remark: m=0 curves}, we construct the space of $m$-stable maps for $m \geq 1$ since our proof uses properness of the Kontsevich space.
\label{remark: m=0 maps}
\end{remark}

We now define $m$-stable maps of class $\beta$.  If $k = \CC$, then we take any non-zero $\beta \in H_2(X,\ZZ)$, and say that a map $\mu: C \ra X$ is of class $\beta$ if $\mu_*([C]) = \beta$.  In general, we require a slight modification of this setup, well-known from \cite{BM}.  We define:
\begin{eqnarray*}
H_2(X) &=& \Hom_\ZZ(\Pic(X), \ZZ)  \\
H_2(X)^+ &=& \{ \beta \in H_2(X): \beta(L) \geq 0 \text{ for any ample line bundle } L \}.
\end{eqnarray*}
Define $\mu_*([C]) \in H_2(X)$ to be the map $L \mapsto \deg \mu^*(L)$; obviously $\mu_*([C]) \in H_2(X)^+$.  For instance, if $\mu: C \ra \PP^r$ is a map of degree $d$, then $\mu_*([C])$ is the map $\ZZ \ra \ZZ$ given by multiplication by $d$; in this case we write $\beta = d$.  We now take any non-zero $\beta \in H_2(X)^+$, and say that a map $\mu: C \ra X$ is of class $\beta$ if $\mu_*([C]) = \beta$.

\begin{definition}
\label{def: class beta}
Let $(C, p_1, \ldots, p_n)$ be a pointed genus one curve over an algebraically closed field $k$, and let $\mu: C \ra X$ be a map to a projective variety $X$.  Fix any non-zero $\beta \in H_2(X)^+$.  Then $\mu$ is said to be an \emph{$m$-stable map of class $\beta$} (or an \emph{$m$-stable map to $(X,\beta)$}) if $\mu$ is an $m$-stable map and $\mu$ is of class $\beta$.
\end{definition}

\section{Construction of $\sMbar_{1,n}^m(X,\beta)$}
\label{section: construction}

In Section \ref{subsection: definition of the moduli problem} we define the moduli problem of $m$-stable maps and show that it is representable by a Deligne-Mumford stack.  We prove the valuative criterion for properness in Section \ref{subsection: properness}.

\subsection{Definition of the moduli problem}
\label{subsection: definition of the moduli problem}
Let $m \geq 1$.  Fix a projective variety $X$ and a non-zero $\beta \in H_2(X)^+$.  For any scheme $S$, we define an \emph{$m$-stable $n$-pointed map to $X$ of class $\beta$ over $S$} to be a flat, proper family of genus one curves $\pi: \sC \ra S$ with $n$ sections $\sigma_1, \ldots, \sigma_n$, together with a map $\mu: \sC \ra X$ such that the restriction of $(\sC, \{ \sigma_i \}, \mu)$ to each geometric fiber is an $m$-stable $n$-pointed map to $X$ of class $\beta$.  A morphism between two $m$-stable maps $(\sC \ra S, \{ \sigma_i \}, \mu_\sC: \sC \ra X)$ and $(\sD \ra T, \{ \tau_i \}, \mu_\sD: \sD \ra X)$ is a Cartesian square
\begin{diagram}
\sC &\rTo^f& \sD  \\
\dTo&		 	 & \dTo	\\
S		&\rTo	 & T		\\
\end{diagram}
\noindent such that $f(\sigma_i) = \tau_i$ and $\mu_\sC = \mu_\sD \circ f$.  This defines a category of $m$-stable maps to $X$ of class $\beta$, denoted
\[
\sMbar_{1,n}^m(X,\beta),
\]
\noindent which is fibered in groupoids over the category of $k$-schemes.  By Lemma \ref{lemma: polarized}, our moduli problem admits a canonical polarization
\[
(\sC \ra S, \{ \sigma_i \}, \mu: \sC \ra X) \mapsto \omega_{\sC/S}(\sigma_1 + \ldots + \sigma_n) \otimes \mu^* \sO_X(3)
\]
\noindent and therefore satisfies descent over $(\Sch/k)$ endowed with the fppf-topology, i.e. $\sMbar_{1,n}^m(X,\beta)$ is a stack.

The main result of Part I of this paper is:

\begin{theorem}
\label{theorem: construction}
$\sMbar_{1,n}^m(X,\beta)$ is a proper Deligne-Mumford stack of finite type over $\Spec \ZZ[1/6]$.
\end{theorem}

Before proving Theorem \ref{theorem: construction}, we verify a lemma (analogous to Lemma 4.2 in \cite{BM}) that will allow us to reduce certain statements about families of $m$-stable maps to the corresponding statements about families of $m$-stable curves.  Specifically, we will use it to derive deformation-openness of $m$-stable maps (Lemma \ref{lemma: def-open}) and the uniqueness portion of the valuative criterion for properness of $\sMbar_{1,n}^m(X,\beta)$ (Section \ref{subsubsection: uniqueness of limits}) from the corresponding results in \cite{S1}.

\begin{lemma}
Assume $\ch k \neq 2,3$.  Let $(\sC, \{\sigma_i\}, \mu_\sC)$ and $(\sD, \{\tau_i\}, \mu_\sD)$ be families of $n$-pointed $m$-stable maps to $X$ over a base $k$-scheme $S$, and let $x \in S(\bar{k})$ be a geometric point of $S$.  Then there exists an \etale \ neighborhood $S' \ra S$ of $x$ and disjoint unmarked smooth sections $\{ \sigma_{j,k} \}$ on $\sC \times_S S'$ and $\{ \tau_{j,k} \}$ on $\sD \times_S S'$, $1 \leq j \leq m+1$, $1 \leq k \leq d$, such that $(\sC \times_S S', \{\sigma_i, \sigma_{j,k}\})$ and $(\sD \times_S S', \{\tau_i, \tau_{j,k}\})$ are $m$-stable $(n + d(m+1))$-pointed curves over $S'$.
\label{lemma: reduction}
\end{lemma}
\begin{proof}
Assume without loss of generality that $\mu_{\sC*}([\sC]) = \mu_{\sD*}([\sD]) =: \beta$.  Choosing an embedding $\iota: X \hra \PP^r$ and letting $d = \iota_*(\beta)$, we reduce to the case $X = \PP^r$ and $\beta = d$.  Since $\mu_\sC$ and $\mu_\sD$ are separable, there exists an \etale \ neighborhood $S' \ra S$ of $x$ and hyperplanes $H_1, \ldots, H_{m+1}$ in $\PP^r$ which intersect the maps $\mu_\sC$ and $\mu_\sD$ transversely at unmarked smooth points, i.e.\ for each $1 \leq j \leq m+1$, the divisors $\mu_\sC^*(H_j)$ and  $\mu_\sD^*(H_j)$ consist of disjoint unmarked smooth sections over $S'$.  Write $\mu_\sC^*(H_j) = \sum \sigma_{j,k}$ and $\mu_\sD^*(H_j) = \sum \tau_{j,k}$, $1 \leq k \leq d$.  Then by Lemmas \ref{lemma: curve automorphisms} and \ref{lemma: map automorphisms},
\[
(\sC \times_S S', \{ \sigma_i, \sigma_{j,k} \})
\]
\noindent and
\[
(\sD \times_S S', \{ \tau_i, \tau_{j,k} \})
\]
\noindent are $m$-stable $(n + d(m+1))$-pointed curves.
\end{proof}

We now proceed to the proof of Theorem \ref{theorem: construction}.  We show properness in Section \ref{subsection: properness}.  We show that our moduli problem is canonically polarized, bounded, and deformation-open in Lemmas \ref{lemma: polarized}, \ref{lemma: bounded}, and \ref{lemma: def-open}, respectively.  The rest of the argument is standard and given below.

\begin{proof}
By Theorem 4.21 in \cite{DM}, it suffices to show that for every $k$-scheme $S$ with $\ch k \neq 2,3$, the functor $\uIsom_S$ is representable by a quasi-compact scheme of finite type and unramified over $S$, and that $\sMbar_{1,n}^m(X,\beta)$ admits a smooth atlas of finite type.  Let $x = (\sC, \{ \sigma_i \}, \mu_\sC)$ and $y = (\sD, \{ \tau_i \}, \mu_\sD)$ be $m$-stable $n$-pointed maps to $X$ of class $\beta$ over a $k$-scheme $S$.  Then $x$ and $y$ are canonically polarized by Lemma \ref{lemma: polarized}.  There is a natural morphism of $S$-schemes $\sC \ra X \times S$.  Applying Grothendieck's results on the representability of the Hilbert scheme to the graph of any base change of this morphism (see, e.g., \cite{K2}, Theorem 1.10), we obtain that the functor $\uIsom_S(x,y)$ is representable by a quasi-projective scheme $\bIsom_S(x,y)$ over $S$.  To check that $\bIsom_S(x,y)$ is unramified, it is enough to check it over an algebraically closed field $k$, where it is automatic since an $m$-stable map has no infinitesimal automorphisms by definition.  This gives the first assertion above.

We now construct a smooth atlas over $\sMbar_{1,n}^m(X,\beta)$ via a standard argument using the Hilbert scheme.  We can obviously reduce to the case $(X,\beta) = (\PP^r,d)$.  Using Lemmas \ref{lemma: bounded} and \ref{lemma: def-open}, the argument is now exactly the same as that in \cite{FP}.  Let $(C, \{p_i\}, \mu)$ be any $m$-stable $n$-pointed map to $\PP^r$, and let $L$ and $N$ be as in Lemma \ref{lemma: bounded}.  Let $D = \deg(L^N)$ and $R = h^0(C,L^N)-1$; by Riemann-Roch, $R=D-1$.  Let
\[
\sH = \sH_{(D,d),1}(\PP^R \times \PP^r)
\]
\noindent be the Hilbert scheme of curves of bidegree $(D,d)$ and arithmetic genus one in $\PP^R \times \PP^r$.  Note that $\sH$ is of finite type since we have fixed $d$.  By Lemma \ref{lemma: bounded}, after choosing a basis for $H^0(C,L^N)$, any $m$-stable $n$-pointed map $\mu: (C, p_1, \ldots, p_n) \ra (\PP^r,d)$ corresponds to a point in
\[
\sH \times (\PP^R \times \PP^r)^n
\]
\noindent via the morphism $(\iota_{\sL^N}, \mu)$.  Define the closed incidence subscheme
\[
I \subset \sH \times (\PP^R \times \PP^r)^n
\]
\noindent of tuples $([C], p_1, \ldots, p_n)$ with $p_1, \ldots, p_n \in C$.  By Lemma \ref{lemma: def-open} and standard results, there is open subscheme $V \subset U$ of $([C], p_1, \ldots, p_n)$ satisfying the following:
\begin{enumerate}
\renewcommand{\labelenumi}{(\roman{enumi})}
  \item The points $p_1, \ldots, p_n$ lie in the smooth locus of $C$.
	\item The natural projection $\pi_1: C \ra \PP^R$ is a non-degenerate embedding.
	\item The natural projection $\pi_2: C \ra \PP^r$ is an $m$-stable $n$-pointed map.
\end{enumerate}
Finally, using representability of the Picard functor of the universal curve $\sU \ra \sH$ (see, e.g., \cite{MFK}, Proposition 5.1), there is a locally-closed subscheme $U \subset V$ of $([C], p_1, \ldots, p_n)$ such that the line bundles $(\omega_C(p_1 + \ldots + p_n) \otimes \pi_2^* \sO_{\PP^r}(3))^N$ and $\pi_1^*\sO_{\PP^R}(1)$ are isomorphic.  Then $U$ is Zariski-locally a $PGL_{R+1}$-torsor over $\sMbar_{1,n}^m(X,\beta)$ (corresponding to the choice of basis of $H^0(C,L^N)$), so it is smooth and surjective over $\sMbar_{1,n}^m(X,\beta)$.
\end{proof}

\begin{lemma}[Canonical polarization]
For any $m$-stable map $\mu: (\sC, \sigma_1, \ldots, \sigma_n) \ra X$ over a scheme $S$, the line bundle
\[
\omega_{\sC/S}(\sigma_1 + \ldots + \sigma_n) \otimes \mu^* \sO_X(3)
\]
\noindent is relatively ample.
\label{lemma: polarized}
\end{lemma}
\begin{proof}
Given an $m$-stable map $\mu: (C, p_1, \ldots, p_n) \ra X$, we need to show that the line bundle
\[
\sL = \omega_C(p_1 + \ldots + p_n) \otimes \mu^* \sO_X(3)
\]
\noindent is ample.  Let $F \subset C$ be any irreducible component.  If $\mu$ is non-constant on $F$, then $\omega_C(p_1 + \ldots + p_n)$ has degree at least $-2$ on $F$ and $\mu^* \sO_X(3)$ has degree at least 3 on $F$, so $\sL$ has positive degree on $F$.  Now suppose $\mu$ is constant on $F$.  Then $\mu^* \sO_X(3)$ has degree 0 on $F$, so $\sL$ has positive degree on $F$ if the genus of $F$ is at least two.  Let $\pi: \tilde{C} \ra C$ be the normalization, and let $\tilde{F}$ be the proper transform of $F$.  We have the following cases:
\begin{itemize}
	\item Suppose $F$ has genus one.  Then by the stability condition $F$ has at least two distinguished points, so $\sL$ has degree at least $2g-2+2 = 2$ on $F$.
	\item Suppose $F$ has genus zero and meets the rest of the curve only nodally (i.e.\ $F$ is not a component of the minimal elliptic subcurve).  Then by the stability condition $F$ has at least three distinguished points, so $\sL$ has degree at least $2g-2+3 = 1$ on $F$.
	\item Suppose $F$ has genus zero and meets the rest of the curve at an elliptic $l$-fold point $p$ for some $l \leq m$ (i.e.\ $F$ is a component of the minimal elliptic subcurve).  Then $F$ has at least one marked point $q$, so by Lemma \ref{lemma: Gorenstein}, $\pi^* (\omega_C|_F) = \omega_{\tilde{F}}(q+2r)$, where $r$ is the unique point on $\tilde{F}$ with $\pi(r) = p$.  Thus, $\sL$ has positive degree on $F$.
\end{itemize}
\noindent So $\omega_C$ is ample on each component $F$ of $C$, and is thus ample.
\end{proof}

\begin{lemma}[Boundedness]
Let $(C, p_1, \ldots, p_n, \mu)$ be any $m$-stable map to $(\PP^r,d)$, and define the line bundle
\[
L := \omega_C(p_1 + \ldots + p_n) \otimes \mu^* \sO_{\PP^r}(3).
\]
\noindent Then $L^N$ is very ample on $C$ for any $N > n + \max(2m,4)$.
\label{lemma: bounded}
\end{lemma}
\begin{proof}
The proof is identical to that in \cite{S1}.
\end{proof}

\begin{lemma}[Deformation-openness]
Let $S$ be a connected Noetherian scheme.  Let $\pi: \sC \ra S$ be any flat family of genus one curves over $S$ with sections $\sigma_1, \ldots, \sigma_n$, and let $\mu: \sC \ra X$ be a morphism.  Let $\beta \in H_2(X)^+$.  Then the set
\[
T := \{ s \in S: \mu_{\bar{s}} \text{ is } m\text{-stable of class } \beta \}
\]
\noindent is Zariski-open in $S$.
\label{lemma: def-open}
\end{lemma}
\begin{proof}
The class of $\mu$ is constant on $S$, so it suffices to show that the set
\[
T := \{ s \in S: \mu_{\bar{s}} \text{ is an } m\text{-stable map} \}
\]
\noindent is Zariski-open in $S$.  Choose any geometric point $x \in S(\bar{k})$.  By Lemma \ref{lemma: reduction}, there exists an \etale \ neighborhood $U_x \hra S$ of $x$ and disjoint unmarked smooth sections $\{ \sigma_{j,k}^x \}$ of $\sC \times_S U_x$ such that
\[
T \times_S U_x = \{ s \in U_x: (\sC_{\bar{s}}, \{ \sigma_i(\bar{s}), \sigma_{j,k}^x(\bar{s}) \} ) \text{ is an } m\text{-stable curve} \}.
\]
\noindent By deformation-openness of $m$-stable curves (\cite{S1}, Lemma 3.10), the morphism
\[
T \times_S U_x \hra U_x
\]
\noindent is an open immersion.  Hence, setting $S' := \sqcup_x U_x$, the induced morphism
\[
T \times_S S' \hra S'
\]
\noindent is open, so $T \hra S$ is open since $S' \ra S$ is fpqc (\cite{EGAIV2}, Proposition 2.7.1).
\end{proof}

\subsection{Valuative criterion for properness}
\label{subsection: properness}
Choosing an embedding $\iota: X \hra \PP^r$ and letting $d = \iota_*(\beta)$, we easily reduce to the case $X = \PP^r$ and $\beta = d$.  We now prove the valuative criterion for properness \cite{LMB} of $\sMbar_{1,n}^m(\PP^r,d)$:

\begin{theorem}[Properness of $\sMbar_{1,n}^m(\PP^r,d)$]
\label{theorem: properness}
\mbox{}
Let $\Delta$ be the spectrum of a discrete valuation ring with algebraically closed residue field $k$ such that $\ch k \neq 2, 3$, and let $\eta \in \Delta$ be the generic point.
\begin{enumerate}
	\item (Existence of $m$-stable limits) If $(\sC, \sigma_1, \ldots, \sigma_n, \mu)|_\eta$ is an $n$-pointed $m$-stable map to $\PP^r$ over $\eta$, then there exists a finite base-change $\Delta' \ra \Delta$ and an $m$-stable map $(\sC' \ra \Delta', \sigma_1', \ldots, \sigma_n', \mu')$ to $\PP^r$ such that
	\[
	(\sC', \sigma_1', \ldots, \sigma_n')|_{\eta'} = (\sC, \sigma_1, \ldots, \sigma_n)|_\eta \times_\eta \eta'
	\]	
	\noindent and
	\[
	(\sC'|_{\eta'} \ra \sC|_\eta \ras{\mu} \PP^r)|_{\eta'} = \mu'|_{\eta'}.
	\]
	
	\item (Uniqueness of $m$-stable limits) Suppose that $(\sC \ra \Delta, \sigma_1, \ldots, \sigma_n, \mu_\sC)$ and $(\sD \ra \Delta, \tau_1, \ldots, \tau_n, \mu_\sD)$ are $m$-stable maps to $(\PP^r, d)$.  Then any isomorphism over the generic fiber
	\[
	(\sC, \sigma_1, \ldots, \sigma_n, \mu_\sC)|_\eta \cong (\sD, \tau_1, \ldots, \tau_n, \mu_\sD)|_\eta
	\]
	\noindent extends to an isomorphism over $\Delta$:
	\[
	(\sC, \sigma_1, \ldots, \sigma_n, \mu_\sC) \cong (\sD, \tau_1, \ldots, \tau_n, \mu_\sD).
	\]
\end{enumerate}
\end{theorem}

\begin{remark}
Let $\sM_{1,n}^m(\PP^r, d) \subset \sMbar_{1,n}^m(\PP^r, d)$ be the open substack of maps of degree $d$ from smooth $n$-pointed genus one curves to $\PP^r$ (from now on we will omit the superscript $m$ since these substacks are obviously all isomorphic).  Since $\sMbar_{1,n}^m(\PP^r, d)$ may be reducible, this substack is not necessarily dense.  In particular, it does not \emph{a priori} suffice to prove the valuative criterion for families with smooth generic fiber, but in Section \ref{subsubsection: smooth generic fiber} we will be able to reduce to this case.  This reduction is necessary in order to perform nodal reduction and apply Lemma \ref{lemma: contraction} in the existence portion of the proof of the valuative criterion.  \\
\end{remark}

\subsubsection{Reduction to the case of families with smooth generic fiber}
\label{subsubsection: smooth generic fiber}
\mbox{}  \\  \\
\noindent \textbf{Reduction (a):} \textit{Reduce to the case where the generic fiber of $\sC$ has no disconnecting nodes.}  \\

\noindent We will need the following obvious observations:
\begin{enumerate}
	\item Let $(C, p_1, \ldots, p_n, \mu)$ be an $m$-stable map.  If $Z \subset C$ is a connected subcurve of arithmetic genus one, then the map
	\[
	( Z, \{ Z \cap \ol{C \backslash Z} \} \cup \{ p_i: p_i \in Z \}, \mu|_Z )
	\]
	is $m$-stable.
	\item Let $(C, p_1, \ldots, p_n, \mu)$ be an $m$-stable map.  If $S \subset C$ is a smooth rational component meeting the rest of $C$ at a disconnecting node, then the map
	\[
	( S, \{ S \cap \ol{C \backslash S} \} \cup \{ p_i: p_i \in S \}, \mu|_S )
	\]
	is Kontsevich stable.
	\item Let $(S, \{ s, p_1, \ldots, p_j \}, \mu_1)$ be a Kontsevich stable map from a genus zero curve, let $(Z, \{ z, p_{j+1}, \ldots, p_n \}, \mu_2)$ be an $m$-stable map from a genus one curve, and suppose that $\mu_1(s) = \mu_2(z)$.  Then the identification of $s$ and $z$ along a disconnecting node $(S \cup Z, p_1, \ldots, p_n, \mu)$ is an $m$-stable map.
\end{enumerate}
\noindent By restricting $\Delta$ to a suitable neighborhood of 0, we may assume without loss of generality that all of the fibers $C_t$, $t \neq 0$, have the same number of irreducible components.  Let $Z_t$ denote the minimal elliptic subcurve of $C_t$, and let $(S_1)_t, \ldots, (S_N)_t$ be the remaining rational irreducible components.  Possibly after finite base-change, we may assume that the monodromy around the central fiber acts trivially on the set of irreducible components of the generic fiber.  We can then write $(\sC, \sigma_1, \ldots, \sigma_n, \mu)|_\eta$ as the union of the $m$-stable map (by observation (1))
\[
\sZ := ( Z_t, \{ Z_t \cap \ol{C_t \backslash Z_t} \} \cup \{ \sigma_i: \sigma_i(t) \in Z_t \}, \mu|_{Z_t} )_{t \neq 0}
\]
\noindent and Kontsevich stable maps (by observation (2))
\[
\sS_j := ( (S_j)_t, \{ (S_j)_t \cap \ol{C_t \backslash (S_j)_t} \} \cup \{ \sigma_i: \sigma_i(t) \in (S_j)_t \}, \mu|_{(S_j)_t} )_{t \neq 0}
\]
\noindent over $\eta$.  Suppose that we can complete $\sZ$ by a unique $m$-stable map
\[
(Z_0, \{ z_i \} \cup \{ \bar{\sigma}_i: \bar{\sigma}_i(0) \in Z_0 \}, \mu_0),
\]
\noindent where $\bar{\sigma}_i$ denotes the Zariski closure of $\sigma_i$ in the completed family.  By properness of the Kontsevich space, we can complete each of the $\sS_j$ by a unique Kontsevich stable map
\[
((S_j)_0, \{ s_{j,i} \} \cup \{ \bar{\sigma}_i: \bar{\sigma_i}(0) \in (\sS_j)_0 \}, \mu_j).
\]
\noindent Applying observation (3) successively $N$ times yields that the map of the union
\[
(Z_0 \cup (S_1)_0 \cup \ldots \cup (S_N)_0, \bar{\sigma}_1, \ldots , \bar{\sigma}_n, \mu)
\]
\noindent is $m$-stable, and thus gives a unique $m$-stable completion of the original family $(\sC, \sigma_1, \ldots, \sigma_n, \mu)|_\eta$.  Hence, it suffices to prove existence and uniqueness of an $m$-stable limit for $\sZ$, an $m$-stable map from a family over $\eta$ of connected arithmetic genus one curves with no disconnecting nodes.  \\

\noindent \textbf{Reduction (b):} \textit{Reduce to the case where the generic fiber of $\sC$ is smooth.}  \\

\noindent Let
\[
S \subset \sMbar_{1,n}^m(\PP^r, d)
\]
\noindent denote the open substack of $m$-stable maps $\mu: C \ra \PP^r$ such that $\omega_C \cong \sO_C$, and let $\bar{S}$ denote the closure of $S$ in $\sMbar_{1,n}^m(\PP^r, d)$.  By Reduction (a) and Lemma \ref{lemma: minimal} it suffices to show existence and uniqueness of an $m$-stable limit of a family of maps with generic fiber in $S$.  By the valuative criterion, this is equivalent to showing that the stack $\bar{S}$ is proper.

We now claim that any map in $S$ is smoothable.  Indeed, let $\mu: C \ra \PP^r$ be a map in $S$, and let $L = \mu^* \sO(1)$.  By Lemma \ref{lemma: smoothable}, it suffices to show that $H^1(C, L) = 0$, or equivalently, by Serre duality and triviality of the dualizing sheaf, $H^0(C, L^\vee) = 0$.  This is clear: by definition $L^\vee$ has non-positive degree on any irreducible component of $C$, and strictly negative degree on at least one irreducible component.  Any section of $L^\vee$ must be identically zero on that irreducible component and constant on every other irreducible component, hence identically zero since $C$ is connected.  So any map in $S$ is smoothable, as claimed.  This gives that $\sM_{1,n}(\PP^r, d)$ is dense in $S$, hence in $\bar{S}$.  So by the valuative criterion, it suffices to show existence and uniqueness of an $m$-stable limit of a family of maps with generic fiber in $\sM_{1,n}(\PP^r, d)$.  \\

\subsubsection{Existence of Limits}
\mbox{}  \\  \\
\noindent We now follow the general method of proof of the valuative criterion in \cite{S1}.  The main differences are the following:  we need to apply nodal reduction in Step 1, rather than semi-stable reduction, in order to obtain a family of maps, and in Step 2, we can only contract the minimal elliptic subcurve of the central fiber if the map is constant on it, so that the $m$-stable reduction process may terminate earlier than in \cite{S1}.  \\

\noindent \textbf{Step 1:} \textit{Apply nodal reduction to $(\sC, \mu)|_\eta$.}  \\

\noindent By the nodal reduction theorem \cite{HM}, there exists a finite base-change $\Delta' \ra \Delta$ and a family of maps
\[
(\sC^n \ra \Delta', \sigma_1', \ldots, \sigma_n', \mu': \sC' \ra \PP^r)
\]
\noindent such that
\begin{enumerate}
	\item $(\sC^n, \sigma_1', \ldots, \sigma_n')|_{\eta'} = (\sC, \sigma_1, \ldots, \sigma_n)|_\eta \times_\eta \eta'.$
	\item $(\sC^n|_{\eta'} \ra \sC|_\eta \ras{\mu} \PP^r)|_{\eta'} = \mu'|_{\eta'}.$
	\item $\sC^n$ has nodal central fiber (and smooth generic fiber by (1) and Section \ref{subsubsection: smooth generic fiber}).
	\item $\sC^n$ is smooth.
\end{enumerate}
\noindent Furthermore, we form the ``minimal'' such reduction by contracting all unmarked $\PP^1$'s on which $\mu$ is constant, so that:
\begin{enumerate}
\setcounter{enumi}{4}
	\item Every rational irreducible component of $C^n$ on which $\mu'$ is constant has at least two distinguished points.
\end{enumerate}
\noindent (Note that smoothness of $\sC^n$ is preserved since we have sequentially contracted (-1)-curves.)  For notational simplicity, we will continue to denote the base by $\Delta$ and the given sections by $\sigma_1, \ldots, \sigma_n$.  \\

\noindent \textbf{Step 2:} \textit{Alternate between blowing up the marked points on the minimal elliptic subcurve and contracting the minimal elliptic subcurve.}  \\

\noindent Let $(\sC_0, \mu_0) := (\sC^n, \mu')$.  If $\mu_0$ is non-constant on $Z_0$, the minimal elliptic subcurve of $C_0$, then we only need to stabilize rational components of $C_0$ and proceed immediately to Step 3.  Now suppose that $\mu_0$ is constant on $Z_0$.  Starting from $(\sC_0, \mu_0)$, we will construct a sequence $(\sC_0, \mu_0), (\sC_1, \mu_1), \ldots, (\sC_t, \mu_t)$ of families of maps to $(\PP^r,d)$ over $\Delta$ satisfying the following conditions:
\begin{enumerate}
\renewcommand{\labelenumi}{(\roman{enumi})}
	\item The special fiber $C_i \subset \sC_i$ is a Gorenstein curve of arithmetic genus one.
	\item The total space $\sC_i$ is regular at every node of $C_i$.
	\item The proper transforms of $\sigma_1, \ldots, \sigma_n$ on $\sC_i$ are contained in the smooth locus of $\pi_i$, so we may consider the central fiber as a map from an $n$-pointed curve $(C_i, p_1, \ldots, p_n)$.
	\item Every irreducible component of $C_i$ on which $\mu_i$ is constant has at least two distinguished points for $i > 0$, and every rational irreducible component of $C_0$ on which $\mu_0$ is constant has at least two distinguished points.
	\item $C_i$ has an elliptic $l_{i-1}$ point $p$, where $l_i$ denotes the level of the central fiber $C_i$ (Definition \ref{def: level}).
	\item $l_i \geq l_{i-1}$ for $i < t$.  Furthermore, $l_i = l_{i-1}$ if and only if each irreducible component of $Z_i$ has exactly two distinguished points, where $Z_i$ is the minimal elliptic subcurve of $C_i$.
	\item The map $\mu_i$ is constant on $Z_i$ for $i < t$ and non-constant for $i = t$.  \\
\end{enumerate}

Note that these conditions are the same as those in \cite{S1} except for (iv), due to the modified $m$-stability condition, and (vii), since we cannot construct $\mu_{i+1}$ unless $\mu_i$ is constant on $Z_i$.  It is useful to outline here how each of these conditions is used in the proof that follows:
\begin{itemize}
	\item Condition (i) implies that $C_i$ has a minimal elliptic subcurve $Z_i$.
	\item Conditions (ii) and (iii) imply that $Z_i$ and $\sigma_1, \ldots, \sigma_n$ are Cartier divisors on $\sC_i$.
	\item Condition (iv) is necessary for the existence of the contractions $q_i$ defined below, and conditions (iv) and (v) are the properties needed to check that the $(C_i, \mu_i)$ eventually become $m$-stable.
	\item Condition (vi) is not strictly necessary for the rest of the proof, but the essential content of its derivation is that $t$, as defined in condition (vii), is finite.
\end{itemize}
We construct the $\sC_i$ via the following diagram of birational morphisms over $\Delta$
\begin{diagram}
							 & 						 & \sB_0 		&			 			  &			  &						  & \sB_1 &	\hspace{0.5in} \cdots\cdots\cdots \hspace{0.5in} & \sB_{t-2} & & & & \sB_{t-1} & &  \\							 
							 & \ldTo^{p_1} &			 		& \rdTo^{q_1} &			  & \ldTo^{p_2} &				& \hspace{0.5in} \cdots\cdots\cdots \hspace{0.5in} &					 & \rdTo^{q_{t-2}} & & \ldTo^{p_{t-1}} & & \rdTo^{q_{t-1}} &  \\
\sC_0 				 &						 & \rDashto	& 						& \sC_1 &	\rDashto		& 			& \hspace{0.5in} \cdots\cdots\cdots \hspace{0.5in} &					 & \rDashto & \sC_{t-1} & & \rDashto & & \sC_t
\end{diagram}
\noindent which commute with the morphisms to $\PP^r$.  The definitions of the $\sB_i, p_i, q_i$ are exactly the same as those in \cite{S1}:  given $\sC_i$ satisfying (i)-(vii) with $i < t$, we define $p_i: \sB_i \ra \sC_i$ to be the blow-up of $\sC_i$ at the points $\{ p_j : p_j \in Z_i \}$, and define $q_i: \sB_i \ra \sC_{i+1}$ to be the contraction of $\tilde{Z}_i$, the proper transform of $Z_i$ in $\sB_i$.  The map $\mu_i$ is constant on $Z_i$, hence on $\tilde{Z}_i$, so it induces a map $\mu_{i+1}: \sC_{i+1} \ra \PP^r$ of degree $d$.

We need to check that the contraction $q_i$ exists.  Consider the line bundle on $\sB_i$ defined by
\[
\sL := \omega_{\sB_i / \Delta}(\tilde{Z}_i + \sigma_1 + \ldots + \sigma_n) \otimes \mu^*\sO_{\PP^r}(2).
\]
\noindent This is indeed a line bundle by the above observations.  Adjunction and Lemma \ref{lemma: minimal} give
\[
\sL|_{\tilde{Z}_i} \cong \omega_{\tilde{Z}_i} \cong \sO_{\tilde{Z}_i}.
\]
\noindent An easy verification gives that $\sL$ has non-negative degree on every irreducible component of $B_i$ by (iv), and that the subcurve $E \subset B_i$ on which $\sL$ has degree zero is
\[
E = \tilde{Z}_i \cup F,
\]
\noindent where $F$ is the union of irreducible components of $B_i$ on which $\mu_i$ is constant which are disjoint from $\tilde{Z}_i$ and have exactly two distinguished points.  Therefore, Lemma \ref{lemma: contraction} applied to $\sL$ gives that $\tilde{Z}_i \cup F$ is a contractible subcurve of the special fiber.  Since $\tilde{Z}_i$ is disjoint from $F$, we may of course contract only $\tilde{Z}_i$; this is the desired $q_i: \sB_i \ra \sC_{i+1}$.

We now need to check that $\sC_{i+1}$ satisfes (i)-(vi) if $\sC_i$ does, and that after finitely many steps we achieve (vii).  Except for condition (iv), the conditions (i)-(vi) do not involve the maps to $\PP^r$, so their verification is the same as that in \cite{S1}.  For condition (iv), each exceptional divisor of $p_i$ has two distinguished points, so each irreducible component of $B_i$ on which $\mu_i$ is constant must have at least two distinguished points:  this follows from assumption (5) in Step 1 for $i=0$, and from the assumption of condition (iv) on $\sC_i$ for $i>0$.  Thus, by definition each irreducible component of $C_{i+1}$ on which $\mu_{i+1}$ is constant must have at least two distinguished points.  Finally, to check that we achieve (vii) after finitely many steps, write out the fundamental decomposition of $C_i$:
\[
C_i = Z_i \cup R_1 \cup \ldots \cup R_k.
\]
\noindent By construction, one irreducible component from each subcurve $R_j$ (namely, the irreducible component which meets $Z_i$) is absorbed into the minimal elliptic subcurve $Z_{i+1} \subset C_{i+1}$.  The map $\mu_0$ must be non-constant on some irreducible component of $\ol{C_0 \backslash Z_0}$ since $d > 0$.  Thus, after finitely many steps, $\mu_t$ will be non-constant on some component of $Z_t$.  \\

\noindent \textbf{Step 3:} \textit{Stabilize to obtain the $m$-stable limit.}  \\

\noindent For each $0 \leq i \leq t$, let $\phi_i: \sC_i \ra \sC$ be the stabilization map which contracts all of the irreducible components $F \subset C_i$ such that $F$ has exactly two distinguished points and $\mu_i$ is constant on $F$.  The contraction $\phi_i$ exists since each such $F$ is a smooth rational curve meeting the rest of the special fiber in one or two nodes and the total space $\sC_i$ is regular around $F$.  The images of the sections $\sigma_1, \ldots, \sigma_n$ on $\sC_i$ lie in the smooth locus of $\sC$, and the map $\mu_i$ descends to map $\mu: \sC \ra \PP^r$ of degree $d$ since $\mu_i$ is constant on each component contracted by $\phi_i$.  Thus, the resulting special fiber $(C, p_1, \ldots, p_n, \mu)$ is a map to $(\PP^r,d)$ from an $n$-pointed curve.

Two situations can arise as we increase $i$: the map $\mu_i$ can become non-constant on the minimal elliptic subcurve $Z_i \subset C_i$ (denoted above as $i=t$), in which case the process cannot continue since $\mu_{i+1}$ cannot be defined, or $Z_i$ can have at least $m+1$ distinguished points (denoted $i=e$), in which case $C_{e+1}$ would have an elliptic $l$-fold point with $l>m$.  We now show that the above map $(C, p_1, \ldots, p_n, \mu)$ is $m$-stable for $i = \min(t,e)$, completing the proof of existence.  We split into two cases, $t-1 < e$ and $t-1 \geq e$:  \\

\noindent \textbf{Case 1:} $l_{t-1} \leq m.$  \\

\noindent In this case we take $i = t$.  To show that $(C, p_1, \ldots, p_n, \mu)$ is an $m$-stable map, we must verify conditions (1)-(3) of Definition \ref{def: m-stable map}.
\begin{enumerate}
	\item \emph{$C$ has only nodes and elliptic $l$-fold points, $l \leq m$, as singularities.}  By conditions (i) and (v) above, $C_t$ has only nodes and an elliptic $l_{t-1}$-fold point as singularities, where $l_{t-1} \leq m$ by assumption.  The contraction $\phi_t$ can only introduce additional nodes on $C$.  \\
	\item \emph{Any connected subcurve of $C$ of arithmetic genus one on which $\mu$ is constant has at least $m+1$ distinguished points.}  Since $\mu$ is non-constant on the minimal elliptic subcurve $Z \subset C$, no such subcurve exists, so this condition is automatic.  \\
	\item \emph{$\mu$ has no infinitesimal automorphisms.}  We of course equivalently check the conditions of Lemma \ref{lemma: map automorphisms}.  Let
	\[
	C_t = Z_t \cup R_1 \ldots \cup R_k
	\]
	be the fundamental decomposition of $C_t$.  Since $\phi_t$ contracts every component of $R_1 \cup \ldots \cup R_k$ with two distinguished points, every component of $\phi_t(R_1) \cup \ldots \cup \phi_t(R_k)$ has at least three distinguished points.  It remains to check the condition on irreducible components of $\phi_t(Z_t)$.
	
If $t=0$, then $Z_t$ is a smooth genus one curve and there is nothing to check.  If $t \geq 1$, $Z_t$ consists of $l_{t-1}$ smooth rational branches meeting in an elliptic $l_{t-1}$-fold point, so condition (iv) states that every component of $Z_t$ on which $\mu$ is constant has at least two distinguished points.  But by definition $\phi_t$ maps $Z_t$ isomorphically onto $\phi(Z_t)$, so the same holds for the minimal elliptic subcurve $\phi(Z_t) \subset C$ as well.  \\
\end{enumerate}

\noindent \textbf{Case 2:} $l_{t-1} > m.$  \\

\noindent In this case we define
\[
e := \min \{ 0 \leq j \leq t-1 : l_j > m \}
\]
\noindent and take $i = e$.  We again verify conditions (1)-(3) of Definition \ref{def: m-stable map}.
\begin{enumerate}
	\item \emph{$C$ has only nodes and elliptic $l$-fold points, $l \leq m$, as singularities.}  The proof is the same as the one in Case 1 by our choice of $e$.  \\
	\item \emph{Any connected subcurve of $C$ of arithmetic genus one on which $\mu$ is constant has at least $m+1$ distinguished points.}  It suffices to check this on the minimal elliptic subcurve $Z \subset C$.  The level of $C_e$ is greater than $m$ by our choice of $e$, and we claim that the level of $C$ is equal to the level of $C_e$.  The verification of this is essentially the same as the one in \cite{S1}.  Let
	\[
	C_e = Z_e \cup R_1 \cup \ldots \cup R_k
	\]
	be the fundamental decomposition of $C_e$.  Order the $R_i$ so that $R_1, \ldots, R_j$ consist entirely of irreducible components with two distinguished points and $\mu_e$ is constant on each $R_i$.  We now simply show that under the isomorphism $Z_e \ra \phi_e(Z_e)$, each of the intersections of $Z_e$ with $R_1, \ldots, R_j$ is replaced by a marked point, so that the number of distinguished points remains the same.  More formally, $\phi_e$ contracts each of $R_1, \ldots, R_j$ to a point (and partially contracts the other $R_i$ which contain irreducible components on which $\mu_e$ is constant), so the fundamental decomposition of $C$ is
	\[
	C = \phi_e(Z_e) \cup \phi_e(R_{j+1}) \cup \ldots \cup \phi_e(R_k),
	\]
	and thus
	\[
	\abs{ \{ \phi_e(Z_e) \cap \ol{C \backslash \phi_e(Z_e)} \} } = \abs { \{ Z_e \cap \ol{C_e \backslash Z_e} \} } - j. 
	\]
	But each $R_1, \ldots, R_j$ is a chain of $\PP^1$'s whose final component contains a marked point, so $\phi_e(R_1), \ldots, \phi_e(R_j)$ are marked points on the minimal elliptic subcurve $\phi_e(Z_e)$, and thus
	\[
	\abs{ \{ p_i: p_i \in \phi_e(Z_e) \} } = \abs{ \{ p_i: p_i \in Z_e \} } + j.
	\]
	Summing the last two equations proves the claim.  \\
	\item \emph{$\mu$ has no infinitesimal automorphisms.}  The same argument as in Case 1 gives the condition on irreducible components of $\phi_e(Z_e)$, and gives that every irreducible component of $\phi(Z_e)$ has at least two distinguished points.  Conditions (1) and (2) imply that some component of $\phi(Z_e)$ must have at least three distinguished points.  \\
\end{enumerate}

\subsubsection{Uniqueness of Limits}
\label{subsubsection: uniqueness of limits}
\mbox{}  \\ \\
\noindent Let $(\sC \ra \Delta, \sigma_1, \ldots, \sigma_n, \mu_\sC)$ and $(\sD \ra \Delta, \tau_1, \ldots, \tau_n, \mu_\sD)$ be $m$-stable maps of degree $d$ to $\PP^r$ over $\Delta$, and suppose we have an isomorphism over the generic fiber
\[
\phi: (\sC, \sigma_1, \ldots, \sigma_n, \mu_\sC)|_\eta \cong (\sD, \tau_1, \ldots, \tau_n, \mu_\sD)|_\eta.
\]
\noindent By Lemma \ref{lemma: reduction}, there exists an \etale \ neighborhood $0 \in \Delta' \subset \Delta$ and disjoint unmarked smooth sections $\{ \sigma_{j,k} \}$ and $\{ \tau_{j,k} \}$ of $\sC|_{\Delta'}$ and $\sD|_{\Delta'}$, $1 \leq j \leq m+1$, $1 \leq k \leq d$, such that
\[
(\sC|_{\Delta'}, \{ \sigma_i, \sigma_{j,k} \})
\]
\noindent and
\[
(\sD|_{\Delta'}, \{ \tau_i, \tau_{j,k} \})
\]
\noindent are $m$-stable $(n + d(m+1))$-pointed curves over $\Delta'$ and $\phi$ sends $\{ \sigma_{j,k} \}$ to $\{ \tau_{j,k} \}$.  Since the moduli space of $m$-stable curves is separated, $\phi$ induces an isomorphism of $m$-stable curves
\[
(\sC, \{\sigma_i, \sigma_{j,k} \}) \ra (\sC', \{ \tau_i, \tau_{j,k} \})
\]
\noindent over $\Delta'$, hence over $\Delta$.  In particular, this gives an isomorphism of the original $n$-pointed curves
\[
(\sC, \{\sigma_i\}) \ra (\sD, \{ \tau_i\})
\]
\noindent over $\Delta$.  We have $\phi \circ \mu_\sD = \mu_\sC$ on $\Delta$ since $(\phi \circ \mu_\sD)|_\eta = \mu_\sC|_\eta$, so $\phi$ is an isomorphism of maps over $\Delta$.  \\

\newpage

\part{Geometric properties}

\noindent From this point on, we only consider the case $(X,\beta) = (\PP^r,d)$.

\section{Stabilization}
\label{section: stabilization}
We observe here that $\sMbar_{1,n}^m(\PP^r,d)$ stabilizes for $m \geq d+n$, and explicitly describe its elements.  We will apply this result in Section \ref{section: smoothness} to prove smoothness of $\sMbar_{1,n}^m(\PP^r,d)$ for $d+n \leq m \leq 5$.

\begin{lemma}
The spaces $\sMbar_{1,n}^m(\PP^r,d)$ are equal for all $m \geq d+n$.  More precisely, the elements of $\sMbar_{1,n}^m(\PP^r,d)$ for $m \geq d+n$ are maps $\mu: C \ra \PP^r$ such that
\begin{enumerate}
\renewcommand{\labelenumi}{(\arabic{enumi}')}
	\item $C$ has nodes and elliptic $l$-fold points, $l \leq d+n$, as singularities.
	\item $\mu$ is non-constant on the minimal elliptic subcurve $Z \subset C$.
	\item $\mu$ has no infinitesimal automorphisms.
\end{enumerate}
\label{lemma: stabilization}
\end{lemma}
\begin{proof}
Let $m \geq d+n$.  By definition, any map satisfying (1')-(3') is $m$-stable, so we need to check that $m$-stability implies (1') and (2').  Let $\mu: C \ra X$ be an $m$-stable map.  Suppose $p \in C$ is an elliptic $l$-fold point.  Let $C_1, \ldots, C_l$ be the connected components of the normalization of $C$ at $p$, and let $B_i \subset C_i$ be the irreducible component containing $p$.  Since $\mu$ has degree $d$, it must be constant on at least $l-d$ of the $B_i$.  By the stability condition, each of the corresponding $C_i$ contains at least one marked point.  Thus, $n \geq l-d$, so $l \leq d+n$.  This gives (1').

To show (2'), suppose that $\mu$ is constant on the minimal elliptic subcurve $Z \subset C$.  Define $E \subset C$ to be the maximal connected genus one subcurve on which $\mu$ is constant; by assumption, $E$ is non-empty.  We then have
\[
\abs{ \{ E \cap \ol{C \backslash E} \} } \leq d
\]
and
\[
\abs{ \{ p_i: p_i \in E \} } \leq n,
\]
so
\[
\abs{ \{E \cap \ol{C \backslash E} \} \cup \{p_i: p_i \in E\} } \leq d+n.
\]
However, $m$-stability asserts that
\[
\abs{ \{E \cap \ol{C \backslash E} \} \cup \{p_i: p_i \in E\} } > m \geq d+n,
\]
a contradiction.  So $\mu$ is non-constant on $Z$, proving (2').
\end{proof}

\section{Irreducible components}
\label{section: irreducible components}

\subsection{The principal component}
\label{subsection: smoothability}
\noindent For any integer $m \geq 0$, consider the open substack
\[
\sM_{1,n}(\PP^r,d) \subset \sMbar_{1,n}^m(\PP^r,d)
\]
\noindent of maps whose domain is a smooth $n$-pointed genus one curve.  In contrast with the higher genus case, $\sM_{1,n}(\PP^r,d)$ is irreducible \cite{VZ1}, and thus its closure is an irreducible component of $\sMbar_{1,n}^m(\PP^r,d)$.  We will call this the \emph{principal component} of $\sMbar_{1,n}^m(\PP^r,d)$.  The main result of this section is Lemma \ref{lemma: main component description}, which gives an explicit description of the principal component in the case $n=0$.  We will not need an analogous description in the case $n>0$.  Note that Lemma \ref{lemma: main component description} was stated in \cite{VZ1} in the case of Kontsevich stable genus one maps.  

\begin{definition}
Let $C$ be a curve, and let $\mu_0: C \ra X$ be a map to a projective variety $X$.  Then $\mu_0$ is said to be \emph{smoothable} if there exists a one-parameter family of pointed curves $\sC \ra \Delta$ with smooth generic fiber and central fiber $C$, and a map $\mu: \sC \ra X$ such that $\mu|_{\sC_0} = \mu_0$.  We call $(\sC, \mu)$ a \emph{smoothing} of $\mu_0$.
\end{definition}

\noindent The principal component of $\sMbar_{1,0}^m(\PP^r,d)$ consists of exactly the smoothable $m$-stable maps.  We have the following general criterion for smoothability of maps from curves to projective varieties:

\begin{lemma}
Let $C$ be a smoothable curve, $X \subset \PP^r$ a projective variety, and $\mu_0: C \ra X$ a map.  Let $L = \mu_0^*\sO_X(1)$, and suppose that $H^1(C,L) = 0$.  Then $\mu_0$ is smoothable.
\label{lemma: smoothable}
\end{lemma}
\begin{proof}
In fact, we prove that $\mu_0$ extends to \emph{any} smoothing $\sC \ra \Delta$ of $C$ after suitably restricting $\Delta$.  By choosing $\Delta$ suitably, we can extend $L$ to a line bundle $\sL$ on $\sC$ by choosing any divisor $D$ on $C$ such that $\sO(D) = L$ and extending $D$ over a neighborhood of 0.  We may also choose $\Delta$ so that $H^1(C_b, \sL_b)=0$ for all $b \in \Delta$.  The theorem on cohomology and base change then states that the base-change homomorphism $(\pi_*\sL)_b \ra H^0(\sC_b, \sL_b)$ is an isomorphism.  In particular, globally generating sections $\alpha_1, \ldots, \alpha_s \in H^0(C,L)$ extend to globally generating sections of $\sL$, and by restricting $\Delta$ once more, these sections determine an embedding into projective space.  Hence $\sL$ is very ample.

The map $\mu_0$ is determined by a basis $\bar{t}$ of a subspace of sections $V \subset H^0(C,L)$ with no common zeroes.  Restricting $\Delta$ to a neighborhood over which $\sL$ is trivial, we can extend $V$ to a sub-vector bundle $\sV \subset \pi_* \sL$ of sections of $\sL$ with no common zeroes, and can extend $\bar{t}$ to a local frame $\bar{\gt}$ of $\sV$.  The triple $(\sL,\sV,\bar{\gt})$ then determines a map $\mu: \sC \ra \PP^r$ such that $\mu|_Y = \mu_0$.
\end{proof}

\noindent The next three lemmas will yield one direction of Lemma \ref{lemma: main component description}.

\begin{lemma}
Let $R$ be a curve of arithmetic genus zero (i.e.\ a nodal tree of $\PP^1$'s), and let $\mu_0: R \ra X$ be a map to a projective variety $X$.  Let $p \in R$ be any smooth point.  Then
\begin{enumerate}
\renewcommand{\labelenumi}{(\alph{enumi})}
	\item $H^0(R, \omega_R(p)) = 0$.
	\item The map $\mu_0$ is smoothable.
\end{enumerate}
\label{lemma: smoothable rational trees}
\end{lemma}
\begin{proof}
For part (a), we induct on the number of irreducible components of $R$.  If $R$ is irreducible, then $\deg \omega_R(p) = -1$ and $H^0(R, \omega_R(p)) = 0$.  Now suppose $R$ has $c$ components, and assume that the lemma holds for all $R$ with fewer than $c$ components.  Let $S$ be the unique irreducible component of $R$ containing $p$, and write
\[
R = S \cup R_1 \cup \ldots \cup R_k
\]
\noindent with $R_1, \ldots, R_k$ rational trees.  Let $p_i = S \cap R_i$.  Then
\[
\omega_R(p)|_{R_i} \cong \omega_{R_i}(p_i),
\]
\noindent so by the inductive hypothesis $H^0(R_i, \omega_R(p)|_{R_i}) = 0$, i.e.\ any section $s$ of $\omega_R(p)$ must vanish on the subcurves $R_1, \ldots, R_k$.  But we also have
\[
\deg \omega_R(p)|_S = -2 + k + 1 = k-1,
\]
\noindent so since $s|_S$ vanishes at the points $p_1, \ldots, p_k$, $s|_S$ must be identically zero.  Thus $s$ is identically zero.  This completes the proof of part (a).

Part (b) follows immediately from part (a), Lemma \ref{lemma: smoothable}, and Serre duality.
\end{proof}

\begin{lemma}
Let $C$ be a Gorenstein curve of arithmetic genus one, and let $Z$ be the minimal elliptic subcurve of $C$.  Let $\mu_0: C \ra X$ be a map to a projective variety $X$, and assume that $\mu_0$ is non-constant on $Z$.  Then the map $\mu_0$ is smoothable.
\label{main component}
\end{lemma}
\begin{proof}
Let $L = \mu_0^* \sO_X(1)$.  By Lemma \ref{lemma: Gorenstein is smoothable}, $C$ is smoothable.  So by Lemma \ref{lemma: smoothable} and Serre duality, it suffices to show that
\[
H^0(C, \omega_C \otimes L^\vee) = 0.
\]
Write out the fundamental decomposition
\[
C = Z \cup R_1 \cup \ldots \cup R_k
\]
\noindent and let $p_i = Z \cap R_i$.  Then
\[
\omega_C|_{R_i} \cong \omega_{R_i}(p_i),
\]
\noindent so by Lemma \ref{lemma: smoothable rational trees}(a), any section of $\omega_C$ must vanish on the $R_i$.  In particular, any section $s$ of $\omega_C \otimes L^\vee$ must vanish on the $R_i$ as well.  Since $\mu_0$ is non-constant on $Z$, we have
\[
\deg ((\omega_C \otimes L^\vee)|_Z) < k,
\]
\noindent so since $s|_Z$ vanishes at the points $p_1, \ldots, p_k$, $s|_Z$ must be identically zero.  Thus $s$ is identically zero.
\end{proof}

\noindent We next specialize to the case where $X = \PP^r$ and $\mu_0$ is $m$-stable.

\begin{lemma}
Let $C$ be a Gorenstein curve of arithmetic genus one, and let $Z$ be the minimal elliptic subcurve of $C$.  Let $\mu_0: C \ra \PP^r$ be an $m$-stable map, and assume that $\mu_0$ is constant on $Z$.  Let $E \subset C$ be the maximal connected genus one subcurve on which $\mu_0$ is constant, and let $E \cap \ol{C \backslash E} = \{ q_1, \ldots, q_l \}$.  Assume that the tangent vectors $v_i := {\mu_0}_*(T_{q_i} \ol{C \backslash E} )$ are linearly dependent in $T_{\mu_0(E)}(\PP^r)$.  Then $\mu_0$ is smoothable.
\label{lemma: smoothing the elliptic m-fold point}
\end{lemma}
\begin{proof}
By Lemma \ref{lemma: Gorenstein is smoothable}, $Z$ is smoothable.  Therefore, since $\mu_0$ is constant on $Z$, we can reduce to the case where $Z$ is a smooth genus one curve.  The curve $C$ is now nodal and of compact type, so the lemma follows from a standard argument using the theory of limit linear series (see also \cite{VZ1}).
\end{proof}

\noindent We can now describe the elements of the principal component of $\sMbar_{1,0}^m(\PP^r,d)$:

\begin{lemma}
An $m$-stable map $\mu_0: C \ra \PP^r$ is smoothable if and only if one of the following two conditions holds:
\begin{enumerate}
	\item $\mu_0$ is non-constant on the minimal elliptic subcurve of $C$.
	\item Assume $\mu_0$ is constant on the minimal elliptic subcurve of $C$, let $E \subset C$ be the maximal connected genus one subcurve on which $\mu_0$ is constant, and let $E \cap \ol{C \backslash E} = \{ q_1, \ldots, q_l \}$.  Then the tangent vectors $v_i := {\mu_0}_*(T_{q_i} \ol{C \backslash E} )$ are linearly dependent in $T_{\mu_0(E)}(\PP^r)$.
\end{enumerate}
\label{lemma: main component description}
\end{lemma}
\begin{proof}
Lemmas \ref{main component} and \ref{lemma: smoothing the elliptic m-fold point} prove one direction.  For the other direction, the argument is the same as that in \cite[Lemma 5.9]{V}.  Let $(\sC,\mu)$ be a smoothing of $\mu_0$, and suppose that $\mu_0$ is constant on the minimal elliptic subcurve of $C$ and that the vectors $w_i := T_{q_i} \ol{C \backslash E}$ are linearly independent.  The map $\mu$ factors through the family $\sC'$ formed by contracting $E$; however, by our assumption, the central fiber of $\sC'$ has arithmetic genus zero while its generic fiber is smooth of genus one, a contradiction.  So the $w_i$ must be linearly dependent, and hence the $v_i = \mu_*(w_i)$ must be dependent as well.
\end{proof}

\subsection{The remaining components}

In this section, we explicitly describe the remaining irreducible components of the spaces $\sMbar_{1,n}^m(\PP^r,d)$.  In the classical case $m=0$, this recovers known results on Kontsevich's space $\sMbar_{1,n}(\PP^r,d)$ (see, e.g., \cite[Section 1.2]{VZ2}).

\subsubsection{The case $n=0$}

Let $d = d_1 + \ldots + d_k$ be any partition of $d$ into unordered positive integers.  Define
\[
U_{d_1, \ldots, d_k} \subset \sMbar_{1,0}^m(\PP^r,d)
\]
to be the locus of maps $\mu: C \ra \PP^r$ such that:
\begin{enumerate}
	\item $C$ is a smooth genus one curve $Z$ with $k$ nodally attached $\PP^1$'s $R_1, \ldots, R_k$.
	\item $\mu|_Z$ is constant (i.e.\ has degree 0) and $\mu|_{R_i}$ is a general map of degree $d_i$ from $\PP^1$ to $\PP^r$.
	\item The tangent vectors to the curves $\mu(R_1), \ldots, \mu(R_k)$ at the point $\mu(Z)$ are linearly independent in $T_{\mu(Z)}\PP^r$.
\end{enumerate}
We claim that $U_{d_1, \ldots, d_k}$ is irreducible.  Consider the map
\[
U_{d_1, \ldots, d_k} \ra \sM_{1,k}
\]
given by $(C, \mu) \mapsto (Z, Z \cap R_1, \ldots, Z \cap R_k)$.  Each fiber of this map is an open subset of the fiber product
\[
\prod_{1 \leq i \leq k} \sMbar_{0,1}(\PP^r, d_i)
\]
over the evaluation map to $\PP^r$, and is thus irreducible, so $U_{d_1, \ldots, d_k}$ is irreducible.

Let $S_{d_1, \ldots, d_k}$ be the closure of $U_{d_1, \ldots, d_k}$ in $\sMbar_{1,0}^m(\PP^r,d)$.  Let $S_0$ be the principal component of $\sMbar_{1,0}^m(\PP^r,d)$.  We then have:

\begin{theorem}
For any positive integers $r$ and $d$ and $m \geq 0$, the irreducible components of $\sMbar_{1,0}^m(\PP^r,d)$ are $S_0$ and the $S_{d_1, \ldots, d_k}$ with $m+1 \leq k \leq \min(r,d)$.
\label{theorem: irreducible components, n=0}
\end{theorem}

\noindent In order to prove Theorem \ref{theorem: irreducible components, n=0}, we need the following preliminary lemma:

\begin{lemma}
Let $C$ be a Gorenstein curve of arithmetic genus one, and let $Z$ be the minimal elliptic subcurve of $C$.  Let $\mu_0: C \ra X$ be a map to a projective variety $X$, and assume that $\mu_0$ is constant on $Z$.  Then there exists a family of maps
\begin{diagram}
\sC = \sZ \cup \sR_1 \cup \ldots \cup \sR_k &\rTo^{\mu}& X  \\
\dTo																				&		 			 &		\\
\Delta																			&		 			 &
\end{diagram}
\noindent with central fiber $(C, \mu_0)$ and generic fiber of the form
\[
Z_\eta \cup S_1 \cup \ldots \cup S_k
\]
\noindent with $Z_\eta$ a smooth genus one curve, $S_i \cong \PP^1$ for all $i$, and $\mu|_\sZ$ constant.
\label{extra components}
\end{lemma}
\begin{proof}
Let
\[
C = Z \cup R_1 \cup \ldots \cup R_k
\]
\noindent be the fundamental decomposition of $C$, and let $p_i = Z \cap R_i$.  For each $1 \leq i \leq k$, apply Lemma \ref{lemma: smoothable rational trees}(b) to $\mu_0|_{R_i}$ to obtain a family of maps
\begin{diagram}
\sR_i 		&\rTo^{\mu_i}& X  \\
\dTo			&		 				 &		\\
\Delta_i	&		 				 &
\end{diagram}
\noindent with generic fiber $\PP^1$ and central fiber $R_i$ such that $\mu_i|_{R_i} = \mu_0|_{R_i}$, and let $\tau_i$ be a section of $\sR_i$ through $p_i$.  Let $\pi: \sZ \ra \Delta_0$ be a smoothing of the curve $Z$, which exists by Lemma \ref{lemma: Gorenstein is smoothable}, and let $\sigma_i$ be a section of $\sZ$ through $p_i$ for each $i$.  We may assume that $\Delta_0$ and all of the $\Delta_i$ are equal after restriction to a suitable neighborhood $\Delta$.  Let $\mu_\sZ: \sZ \ra X$ be the map which is constant on each fiber $Z_t$ of $\sZ$ with value $\tau_i(t)$; precisely, $\mu_\sZ = \mu_i \circ \tau_i \circ \pi$.  Then the nodal identification of the sections $\sigma_i$ and $\tau_i$ yields the desired domain $\sC := \sZ \cup \sR_1 \cup \ldots \cup \sR_k$, and by definition the maps $\mu_\sZ$ and $\mu_i$ patch together to a map $\mu: \sC \ra X$ constant on $\sZ$.
\end{proof}

\begin{proof}[Proof of Theorem \ref{theorem: irreducible components, n=0}]
We first claim that $U_{d_1, \ldots, d_k}$ is open (possibly empty) in $\sMbar_{1,0}^m(\PP^r,d)$ for any partition of $d$.  Suppose $U_{d_1, \ldots, d_k}$ is non-empty, and consider a family of maps
\[
\mu: \sC \ra \PP^r
\]
\noindent such that the central fiber $\mu_0: C \ra \PP^r$ lies in $U_{d_1, \ldots, d_k}$.  Let $\mu_\eta: C_\eta \ra \PP^r$ denote the generic fiber.  By conditions (2) and (3) and Lemma \ref{lemma: main component description}, $\mu_0$ is not smoothable.  So by condition (1), $C_\eta$ must be an elliptic curve $Z$ with $l$ nodally attached $\PP^1$'s $Q_1, \ldots, Q_l$ for some $l$.  Each $Q_i$ must specialize to an $R_j$ on which the map has the same degree, so $l=k$ and $\mu_\eta$ has degree $d_i$ on $Q_i$.  Thus, $\mu_\eta$ satisfies conditions (1) and (2).  Assuming conditions (1) and (2), condition (3) is open, so $\mu_\eta$ satisfies condition (3) as well.  Hence, $U_{d_1, \ldots, d_k}$ is open.

By definition, $U_{d_1, \ldots, d_k}$ is empty for $k > d$; by condition (3), $U_{d_1, \ldots, d_k}$ is empty for $k > r$; and by the stability condition, $U_{d_1, \ldots, d_k}$ is empty for $k \leq m$.  It is clear that the remaining $U_{d_1, \ldots, d_k}$ are non-empty.  So for $m+1 \leq k \leq \min(r,d)$, $S_{d_1, \ldots, d_k}$ is an irreducible component of $\sMbar_{1,0}^m(\PP^r,d)$.

We next verify that the components $S_0, S_{d_1, \ldots, d_k}$ are distinct.  If two components $S_{d_1, \ldots, d_k}$ and $S_{e_1, \ldots, e_l}$ coincided, we could find a family of general elements of $S_{d_1, \ldots, d_k}$ specializing to a general element of $S_{e_1, \ldots, e_l}$.  As above, this would imply that $k=l$ and $d_i = e_i$ for all $i$.  By conditions (2) and (3) and Lemma \ref{lemma: main component description}, none of the components $S_{d_1, \ldots, d_k}$ can coincide with $S_0$.

We now show that every element $\mu: C \ra \PP^r$ of $\sMbar_{1,0}^m(\PP^r,d)$ lies in at least one of the above components.  Let $C = Z \cup R_1 \cup \ldots \cup R_k$ be the fundamental decomposition of the domain curve.  If either $\mu$ is non-constant on $Z$, or $\mu$ is constant on $Z$ and $k>r$, then $\mu$ lies in $S_0$ by Lemma \ref{lemma: main component description}.  Suppose that $\mu$ contracts $Z$ and $k \leq r$.  Define
\[
V_{d_1, \ldots, d_k}
\]
\noindent to be the locus of maps $\mu: C \ra \PP^r$ satisfying only conditions (1) and (2) above.  Note that $U_{d_1, \ldots, d_k}$ is dense in $V_{d_1, \ldots, d_k}$.  Then by Lemma \ref{extra components}, $\mu$ is the limit of a family of maps in some $V_{d_1, \ldots, d_k}$, hence of a family of maps in $U_{d_1, \ldots, d_k}$, and thus lies in $S_{d_1, \ldots, d_k}$.
\end{proof}

\subsubsection{The case $n>0$}

Let $d = d_1 + \ldots + d_k$ be any partition of $d$ into unordered positive integers, and let $0 \leq j \leq n$ be an integer.  Define
\[
U_{d_1, \ldots, d_k, j} \subset \sMbar_{1,n}^m(\PP^r,d)
\]
to be the locus of $n$-pointed $m$-stable maps $\mu: C \ra \PP^r$ satisfying conditions (1)-(3) in the prevous subsection, together with the additional condition that
\begin{enumerate}
\setcounter{enumi}{3}
	\item $j$ of the marked points lie on $Z$.
\end{enumerate}
There are ${n \choose j}$ ways to choose which marked points lie on $Z$, and $k^{n-j}$ ways to distribute the remaining marked points among the $\PP^1$'s $R_1, \ldots, R_k$.  Denote the resulting ${n \choose j} k^{n-j}$ open subsets of $U_{d_1, \ldots, d_k, j}$ by
\[
U_{d_1, \ldots, d_k, j, \alpha}, \ \ 1 \leq \alpha \leq {n \choose j} k^{n-j}.
\]
The same argument as in the case $n=0$ gives that each $U_{d_1, \ldots, d_k, j, \alpha}$ is irreducible.  However, unlike the case $n=0$, the $U_{d_1, \ldots, d_k, j, \alpha}$ need not be distinct:  indeed, if two of the $d_i$ are equal, permuting the marked points on corresponding $R_i$ does not change the map.

Let $S_{d_1, \ldots, d_k, j, \alpha}$ be the closure of $U_{d_1, \ldots, d_k, j, \alpha}$ in $\sMbar_{1,n}^m(\PP^r,d)$.  Let $S_0$ be the principal component of $\sMbar_{1,n}^m(\PP^r,d)$.  We then have:
\begin{theorem}
For any positive integers $r$ and $d$, $n \geq 0$, and $m \geq 0$, the irreducible components of $\sMbar_{1,n}^m(\PP^r,d)$, possibly with repetitions, are $S_0$ and the $S_{d_1, \ldots, d_k, j, \alpha}$ with $m+1 \leq k+j$ and $k \leq \min(r,d)$.
\label{theorem: irreducible components, n>0}
\end{theorem}
\begin{proof}
By definition, $U_{d_1, \ldots, d_k}$ is empty for $k > d$; by condition (3), $U_{d_1, \ldots, d_k}$ is empty for $k > r$; and by the stability condition, $U_{d_1, \ldots, d_k, j}$ is empty for $k + j \leq m$.  It is clear that the remaining $U_{d_1, \ldots, d_k, j, \alpha}$ are non-empty.  The remainder of the proof is the same as that of Theorem \ref{theorem: irreducible components, n=0}.
\end{proof}

\begin{corollary}
For $m \geq \min(r,d) + n$, $\sMbar_{1,n}^m(\PP^r,d)$ is irreducible.
\label{eventual irreducibility, n>0}
\end{corollary}

\section{Smoothness}
\label{section: smoothness}

In \cite{S2}, Smyth (indirectly) analyzes the deformation theory of elliptic $m$-fold points, and proves the following result:

\begin{lemma}
Let $C$ be a curve, and let $p \in C$ be an elliptic $m$-fold point.  Then the miniversal deformation space $\Def(p) = \Def(\sO_{C,p})$ is smooth if and only if $m \leq 5$.
\label{lemma: deformation space}
\end{lemma}

\begin{corollary}
\label{cor: smoothness for curves}
$\sMbar_{1,n}(m)$ is smooth if and only if $m \leq 5$.  At any point $[C]$ such that $C$ has an elliptic $l$-fold point $p$ with $l > 5$, the stack $\sMbar_{1,n}(m)$ is \etale \ locally a smooth fibration over the miniversal deformation space of $p$.
\end{corollary}
\begin{proof}
For any complete, reduced, pointed curve $(C, p_1, \ldots, p_n)$ of finite type over an algebraically closed field $k$ and point $p \in C$, it is well-known that the map of miniversal deformation spaces
\[
\Def(C) \ra \Def(p)
\]
\noindent is smooth, and it is clear that
\[
\Def(C,p_1, \ldots, p_n) \ra \Def(C)
\]
\noindent is smooth since the marked points $p_1, \ldots, p_n$ are smooth.  The statement immediately follows.
\end{proof}

By Corollary \ref{eventual irreducibility, n>0}, the stack $\sMbar_{1,n}^m(\PP^r,d)$ is irreducible for $m \geq \min(r,d) + n$.  While this stack may be singular in general, we can say more in the ``limit'' $m \geq d+n$ (see Lemma \ref{lemma: stabilization}):

\begin{theorem}
\label{theorem: smoothness}
The stack $\sMbar_{1,n}^{d+n}(\PP^r,d)$ is smooth if $d+n \leq 5$.
\end{theorem}
\begin{proof}
Suppose $d+n \leq 5$, and let $\mu: C \ra \PP^r$ be an $(d+n)$-stable map.  We can obviously reduce to the case $n=0$ since the marked points are smooth.  To show that $\sMbar_{1,0}^d(\PP^r,d)$ is smooth at $[\mu]$, it suffices to show that the obstruction space $\Obs(\mu)$ vanishes.  Let $\Obs_R(\mu)$ denote the obstruction space for \emph{rigid} deformations of $\mu$; that is, deformations of $\mu$ which leave both $C$ and $X$ fixed.  Consider the end of the tangent-obstruction exact sequence
\[
\Obs_R(\mu) \ra \Obs(\mu) \ra \Obs(C) \ra 0.
\]
\noindent Since $d+n \leq 5$, Lemma \ref{lemma: deformation space} gives $\Obs(C)=0$, so it suffices to show that $\Obs_R(\mu)=0$.  It is well-known that $\Obs_R(\mu) = H^1(C, \mu^* T_{\PP^r})$.  Pulling back the Euler sequence on $\PP^r$ via $\mu$ and taking cohomology yields a surjection
\[
H^1(C, \mu^* \sO_{\PP^r}(1))^{r+1} \ra H^1(C, \mu^* T_{\PP^r}) \ra 0.
\]
By Lemma \ref{lemma: stabilization}, $\mu$ is non-constant on the minimal elliptic subcurve of $C$, and thus $H^1(C, \mu^* \sO_{\PP^r}(1)) = 0$ by the proof of Lemma \ref{main component}.  So $\Obs_R(\mu) = 0$, and thus $\Obs(\mu)=0$.
\end{proof}

\section{An example: plane cubics}
\label{section: plane cubics}
The spaces $\sMbar_{1,0}^m(\PP^2,3)$ give an illuminating illustration of how the spaces $\sMbar_{1,n}^m(X,\beta)$ ``evolve'' as $m$ increases.  We describe these spaces in this section, starting with the classical case $m=0$, which is one of the well-understood Kontsevich spaces in genus greater than zero.  \\

\noindent $\mathbf{\underline{m=0.}}$ (Figure \ref{figure: m=0})
We recount some well-known facts about the Kontsevich space $\sMbar_{1,0}(\PP^2,3)$.  The space $\sMbar_{1,0}(\PP^2, 3)$ has three irreducible components $S_0, S_1, S_2$ of dimensions 9, 10, 9 which are the closures of the following loci $U_0, U_1, U_2$:
\begin{enumerate}
	\item[($U_0$)] Maps $\mu: Z \ra \PP^2$ from a smooth genus one curve with image a smooth plane cubic,
	\item[($U_1$)] Maps $\mu: Z \cup R \ra \PP^2$ from a smooth genus one curve with a rational tail such that $\mu(Z)$ is a point and $\mu(R)$ is a nodal plane cubic through that point,
	\item[($U_2$)] Maps $\mu: Z \cup R_1 \cup R_2 \ra \PP^2$ from a smooth genus one curve with two rational tails such that $\mu(Z)$ is a point and $\mu(R_1)$ and $\mu(R_2)$ are a plane conic and a line meeting transversely at that point.
\end{enumerate}
We verify the dimensions of these components by determining the dimensions of the $U_i$.  The first ``principal'' component has dimension 9 since $U_0$ is isomorphic to the open subset in $\PP^9$ of smooth plane cubics.  For $U_1$, there is a 1-dimensional space of smooth genus one curves with one marked point ($\dim \sM_{1,1} = 1$), an 8-dimensional family of nodal plane cubics, and a 1-dimensional space of possible images of $Z$, for a total of 10 dimensions.  For $U_2$, there is a 2-dimensional space of smooth genus one curves with two marked points ($\dim \sM_{1,2} = 2$), a 5-dimensional space of plane conics, and a 2-dimensional space of plane lines (and a 0-dimensional space of possible images of $Z$, namely 2 points), for a total of 9 dimensions.

We will now describe the generic element of the intersections $S_0 \cap S_1$ and $S_0 \cap S_2$, the loci which will be modified at the stages $m=1$ and $m=2$, respectively.  Consider a family of maps in $\sMbar_{1,0}^m(\PP^2,3)$
\begin{diagram}
\sC 		&\rTo^{\mu}& \PP^2  \\
\dTo		&		 			 &				\\
\Delta	&		 			 &
\end{diagram}
\noindent with generic fiber in $U_0$ and central fiber of the form $C_0 = Z \cup R$ such that $\mu_0$ is constant on $Z$ and has degree 3 on $R$.  Let $\mu$ be given by sections of a line bundle $\sL$ on $\sC$.  Observe that $h^0(\sL|_{C_t})$ jumps at $t=0$:  indeed, by Riemann-Roch, $h^0(\sL|_{C_\eta}) = 3$ but $h^0(\sL|_{C_0}) = 4$.  Extend the space of sections of $\sL|_{\sC \backslash C_0}$ to a 3-dimensional subspace $V \subset h^0(\sL|_{C_0})$.  Then we claim that the map $\mu_0: C_0 \ra \PP^2$ associated to the pair $(\sL|_{C_0}, V)$ has \emph{cuspidal} image.  Indeed, let $L'$ be a line bundle on $C_0$ with degree 3 on $Z$ and degree 0 on $R$, and let $V' = H^0(C_0, \sL')$.  Then one trivially checks that the vanishing sequence of $V'$ at $p := Z \cap R$ is given by $(0,1,3)$.  By the basic theory of limit linear series \cite{HM}, the vanishing sequence of $V$ must then be at least $(0,2,3)$, so any regular function on $\mu_0(C_0)$ vanishing at $\mu_0(p)$ must also have vanishing first derivative.  This precisely means that $\mu_0(p)$ is a cusp of $\mu_0(C_0)$, as claimed.

Conversely, any such $\mu_0: Z \cup R \ra \PP^2$ is smoothable.  To see this, let $\sC' \ra \Delta$ be a general pencil of plane cubic curves with smooth generic fiber and cuspidal central fiber.  Let $q = \mu_0(p)$.  Partial stable reduction of $\sC'$ gives a rational map
\[
\mu: \sC --> \sC'
\]
\noindent where $\sC$ is a family of genus one curves with smooth generic fiber and central fiber of the form $\tilde{Z} \cup R$ with $\tilde{Z}$ a smooth genus one curve, $R$ a rational tail, and $\tilde{Z} \cap R = \mu_0(p)$.  By choosing the pencil $\sC'$ suitably, we can ensure that $\tilde{Z}$ has the same $j$-invariant as $Z$.  Furthermore, the map $\mu$ is in fact regular:  its only possible points of indeterminancy are the centers of blow-ups which lie on $\tilde{Z}$ away from $\mu_0(p)$, and since $\mu(\tilde{Z})$ is a point, $\mu$ must be regular.  By construction $\mu|_C = \mu_0$, so $\mu$ is the desired smoothing.

Thus, the generic element of $S_0 \cap S_1$ is of the form $\mu_0$ above:  that is, a map $\mu: Z \cup R \ra \PP^2$ from a smooth genus one curve with a rational tail such that $\mu(Z)$ is a point and $\mu(R)$ is a plane cubic with a \emph{cusp} at that point.  Similarly, one can check that the generic element of $S_0 \cap S_2$ is a map $\mu: Z \cup R_1 \cup R_2 \ra \PP^2$ from a smooth genus one curve with two rational tails such that $\mu(Z)$ is a point and $\mu(R_1)$ and $\mu(R_2)$ are a plane conic and a line intersecting at a \emph{tacnode} at that point.

Finally, there is one more subset of $\sMbar_{1,0}(\PP^2,3)$ which will be modified at the stage $m=3$.  Let $T_0$ denote the locus of maps $\mu: Z \cup R_1 \cup R_2 \cup R_3 \ra \PP^2$ from a smooth genus one curve with three rational tails such that $\mu(Z)$ is a point and $\mu(R_1)$, $\mu(R_2)$, and $\mu(R_3)$ are lines meeting at a planar triple point.  This does not define another irreducible component of $\sMbar_{1,0}^m(\PP^2,3)$:  indeed, by Lemma \ref{lemma: smoothing the elliptic m-fold point}, $T_0 \subset S_0$.  \\

\noindent $\mathbf{\underline{m=1.}}$ (Figure \ref{figure: m=1})
The space $\sMbar_{1,0}^1(\PP^2,3)$ has two irreducible components, $S_0$ and $S_2$, the closures of $U_0$ and $U_2$ defined exactly as in the case $m=0$.  The subset $U_1$ has been entirely removed since 1-stability requires that $Z$ have at least 2 distinguished points.  The generic element of the intersection $S_0 \cap S_1$ is replaced by its 1-stable reduction.  This contracts $Z$ to a cuspidal point by Lemma \ref{lemma: contraction}, so the resulting maps are isomorphisms from a cuspidal genus one curve to a cuspidal plane cubic.  \\

\noindent $\mathbf{\underline{m=2.}}$ (Figure \ref{figure: m=2})
The space $\sMbar_{1,0}^2(\PP^2,3)$ is irreducible.  Similarly to the case $m=1$, the subset $U_2$ has been removed, and the generic element of the intersection $S_0 \cap S_2$ is replaced by its 2-stable reduction, an isomorphism from two $\PP^1$'s intersecting at a tacnode to a plane conic and a line intersecting at a tacnode.  \\

\noindent $\mathbf{\underline{m=3.}}$ (Figure \ref{figure: m=3})
The space $\sMbar_{1,0}^3(\PP^2,3)$ is irreducible and smooth, and is \emph{not} isomorphic to $\sMbar_{1,0}^2(\PP^2,3)$.  Indeed, each element of $T_1$ is replaced by its 3-stable reduction.  The resulting maps are isomorphisms from three $\PP^1$'s intersecting at a planar triple point to three lines in $\PP^2$ intersecting at a planar triple point.

\newpage

\setcounter{figure}{-1}

\begin{figure}[h]
  \centering
\includegraphics[width=3.5in]{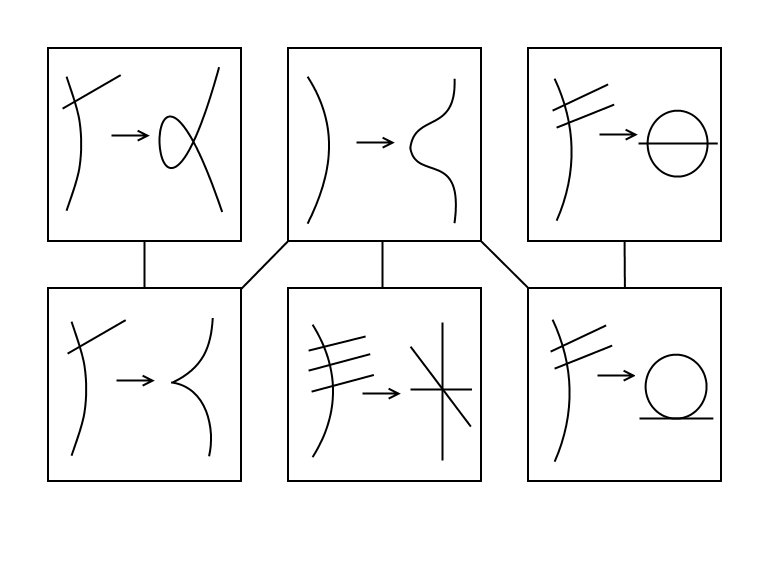}
\caption{$\sMbar_{1,0}(\PP^2,3)$.  The top row depicts the generic element of the irreducible components $S_1,S_0,S_2$.  The bottom left and bottom right depict the generic elements of $S_0 \cap S_1$ and $S_0 \cap S_2$.  The bottom middle depicts elements of the locus $T_0 \subset S_0$.  \label{figure: m=0}}
\end{figure}
\mbox{}
\begin{figure}[h!]
  \centering
\includegraphics[width=3.5in]{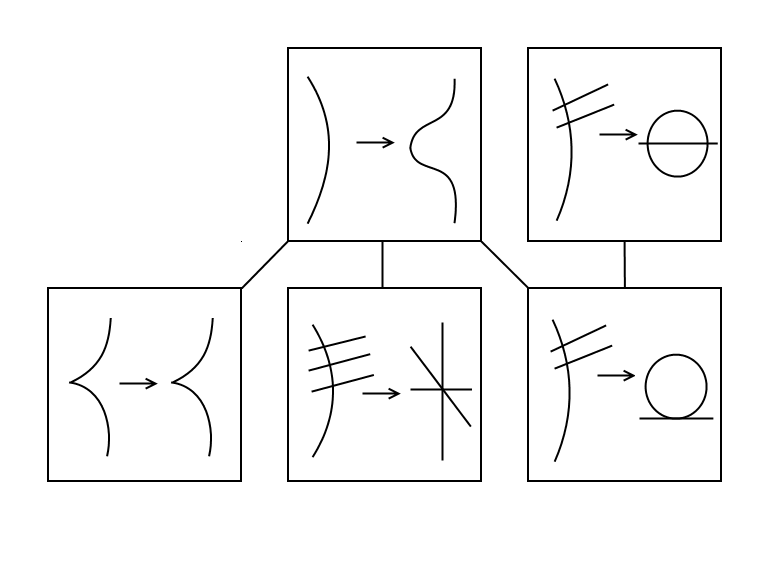}
  \caption{$\sMbar_{1,0}^1(\PP^2,3)$.  The top row depicts the generic element of the irreducible components $S_0,S_2$.  The bottom left and bottom right depict the 1-stable replacement of the generic element of $S_0 \cap S_1$ and the generic element of $S_0 \cap S_2$.  The bottom middle depicts elements of the locus $T_0 \subset S_0$.  \label{figure: m=1}}
\end{figure}
\newpage
\begin{figure}
  \centering
\includegraphics[width=3.5in]{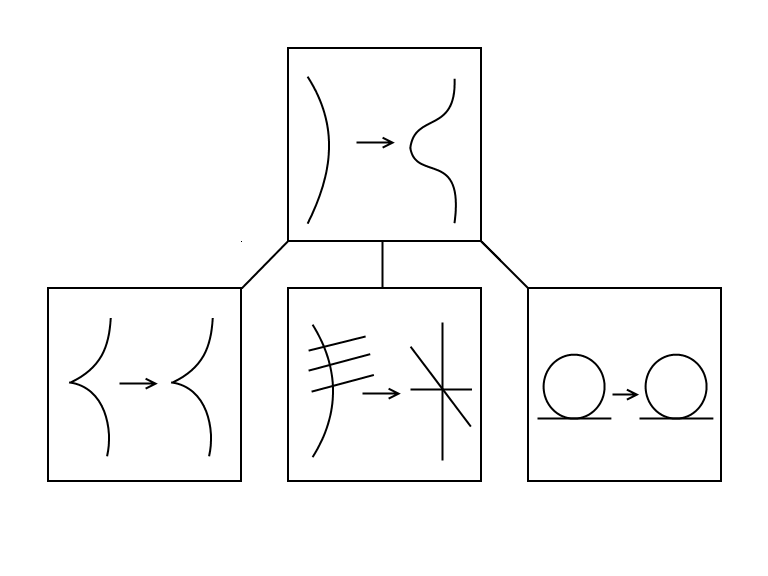}
  \caption{$\sMbar_{1,0}^2(\PP^2,3)$.  The top row depicts the generic element of the irreducible component $S_0$.  The bottom left and bottom right depict the 2-stable replacements of the generic elements of $S_0 \cap S_1$ and $S_0 \cap S_2$.  The bottom middle depicts elements of the locus $T_0 \subset S_0$.  \label{figure: m=2}}
\end{figure}
\mbox{}
\begin{figure}
  \centering
\includegraphics[width=3.5in]{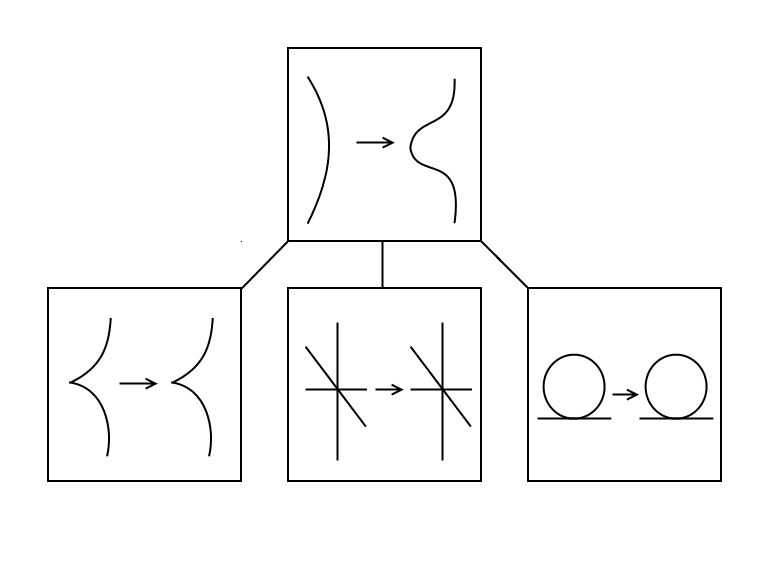}
  \caption{$\sMbar_{1,0}^3(\PP^2,3)$.  The top row depicts the generic element of the irreducible component $S_0$.  The bottom left and bottom right depict the 3-stable replacements of the generic elements of $S_0 \cap S_1$ and $S_0 \cap S_2$.  The bottom middle depicts the 3-stable replacements of elements of the locus $T_0 \subset S_0$.  \label{figure: m=3}}
\end{figure}


\newpage

\begin{thebibliography}{ABCDEF}

\scriptsize

\bibitem[AG]{AG}
V.\ Alexeev and G.\ M.\ Guy. \textit{Moduli of weighted stable maps and their gravitational descendants}, J.\ Inst.\ Math.\ Jussieu 7 (2008), no.\ 3, 425-456.

\bibitem[A]{A}
M.\ Artin. \textit{Versal deformations and algebraic stacks}, Invent.\ Math.\ 27 (1974), 165-189.

\bibitem[BaM]{BaM}
A.\ Bayer and Yu.\ Manin. \textit{Stability conditions, wall-crossing and weighted Gromov-Witten invariants}, Mosc.\ Math.\ J.\ 9 (2009), no.\ 1, 3-32.

\bibitem[BM]{BM}
K.\ Behrend and Yu.\ Manin. \textit{Stacks of stable maps and Gromov-Witten invariants}, Duke Math. J. 85 (1996), no.\ 1, 1-60.

\bibitem[DM]{DM}
P.\ Deligne and D.\ Mumford. \textit{The irreducibility of the space of curves of given genus}, Inst.\ Hautes \'{E}tudes Sci.\ Publ.\ Math.\ (1969) no.\ 36, 75-109.

\bibitem[FP]{FP}
W.\ Fulton and R.\ Pandharipande. \textit{Notes on stable maps and quantum cohomology}, Algebraic geometry - Santa Cruz 1995, 45-96, Proc.\ Sympos.\ Pure Math., 62, Part 2, Amer.\ Math.\ Soc., Providence, RI, 1997.

\bibitem[EGAIV2]{EGAIV2}
A.\ Grothendieck. \textit{\'{E}l\'{e}ments de g\'{e}ometrie alg\'{e}brique. IV. \'{E}tude locale des sch\'{e}mas et des morphismes de sch\'{e}mas. II}, Inst.\ Hautes \'{E}tudes Sci.\ Publ.\ Math.\ (1965), no.\ 24, 231.

\bibitem[HM]{HM}
J.\ Harris and I.\ Morrison. \textit{Moduli of Curves}, Springer GTM 187, 1998.

\bibitem[H]{H}
B.\ Hassett. \textit{Moduli spaces of weighted pointed stable curves}, Adv.\ Math.\ 173 (2003) no.\ 2, 316-352.

\bibitem[HH]{HH}
B.\ Hassett and D.\ Hyeon. \textit{Log canonical models for the moduli space of curves:\ the first divisorial contraction}, Trans.\ Amer.\ Math.\ Soc.\ 361 (2009) no.\ 8, 4471-4489.

\bibitem[Ki]{Ki}
B.\ Kim. \textit{Logarithmic stable maps}, arXiv:0807.3611v2.

\bibitem[KKO]{KKO}
B.\ Kim, A.\ Kresch, and Y.-G.\ Oh. \textit{A compactification of the space of maps from curves}, preprint.

\bibitem[K1]{K1}
J.\ Koll\'{a}r. \textit{Projectivity of complete moduli}, J.\ Diff.\ Geom.\ 32 (1990), 235-268.

\bibitem[K2]{K2}
J.\ Koll\'{a}r. \textit{Rational Curves on Algebraic Varieties}, Ergebnisse der Mathematik und ihrer Grenzgebiete.\ 3.\ Folge, 32.\ Springer-Verlag, Berlin, 1996.

\bibitem[Ko]{Ko}
M.\ Kontsevich. \textit{Enumeration of rational curves via torus actions}, The Moduli Space of Curves (Texel Island, 1994), 335-368, Progr.\ Math.\, 129, Birkh\"{a}user Boston, Boston, MA, 1995.

\bibitem[LMB]{LMB}
G.\ Laumon and L.\ Moret-Bailly. \textit{Champs alg\'{e}briques}, Ergebnisse der Mathematik und ihrer Grenzgebiete, 39. Springer-Verlag, Berlin, 1996.

\bibitem[MOP]{MOP}
A.\ Marian, D.\ Oprea, and R.\ Pandharipande. \textit{The moduli space of stable quotients}, arXiv:0904.2992v2.

\bibitem[MFK]{MFK}
D.\ Mumford, J.\ Fogarty, F.\ Kirwan. \textit{Geometric Invariant Theory}, Ergebnisse der Mathematik und ihrer Grenzgebiete, 34. Springer-Verlag, Berlin, 1991.

\bibitem[MM]{MM}
A.\ \Mustata \ and A.\ \Mustata. \textit{The Chow ring of $\Mbar_{0,m}(\PP^n,d)$}, J.\ Reine Angew.\ Math.\ 615 (2008), 93-119.

\bibitem[Sch]{Sch}
D.\ Schubert. \textit{A new compactification of the moduli space of curves}, Compositio Math.\ 78.\ (1991), no.\ 3, 297-313.

\bibitem[S1]{S1}
D.\ Smyth. \textit{Modular compactifications of $\sM_{1,n}$ I}, arXiv:0808.0177v2.

\bibitem[S2]{S2}
D.\ Smyth. \textit{Modular compactifications of $\sM_{1,n}$ II}, preprint.

\bibitem[S3]{S3}
D.\ Smyth. \textit{Towards a classification of modular compactifications of $\sM_{g,n}$}, arXiv:0902.3690v1.

\bibitem[V]{V}
R.\ Vakil. \textit{The enumerative geometry of rational and elliptic curves in projective space}, J.\ Reine Angew.\ Math.\ 529 (2000), 101-153.

\bibitem[VZ1]{VZ1}
R.\ Vakil and A.\ Zinger. \textit{A natural smooth compactification of the space of elliptic curves in projective space}, Electron.\ Res.\ Announc.\ Amer.\ Math.\ Soc.\ 13 (2007), 53-59.

\bibitem[VZ2]{VZ2}
R.\ Vakil and A.\ Zinger. \textit{A desingularization of the main component of the moduli space of genus-one stable maps into $\PP^n$}, Geom.\ Topol.\ 12 (2008), no.\ 1, 1-95.

\end{thebibliography}
\end{document}